\documentclass[a4paper,12pt]{amsart}
\usepackage[left=3.4 cm,top=2.4cm,bottom=3.4cm,right=3.4cm]{geometry}
\usepackage[utf8]{inputenc}
\usepackage{amsmath}
\usepackage{amssymb}
\usepackage{amsthm}
\usepackage{bbm}
\usepackage{xspace}
\usepackage{mathrsfs}
\usepackage[shortlabels]{enumitem}
\usepackage{comment}
\usepackage{mathtools}
\usepackage{accents}
\usepackage{nicefrac}

\usepackage{tikz-cd}

\newtheorem{definition}{Definition}[section]
\newtheorem{proposition}[definition]{Proposition}
\newtheorem{corollary}[definition]{Corollary}
\newtheorem{theorem}[definition]{Theorem}
\newtheorem{lemma}[definition]{Lemma}

\theoremstyle{definition}

\theoremstyle{remark}

\newtheorem{remark}[definition]{Remark}

\newcommand{\ip}[2]{\left\langle#1,#2\right\rangle}

\newcommand{\M}{\mathcal{M}}
\renewcommand{\Re}{\operatorname{Re}}

\newcommand{\dom}{\Dom}
\newcommand{\Inn}{\operatorname{Inn}}
\renewcommand{\Im}{\operatorname{Im}}
\newcommand{\Std}{{\rm Std}\xspace}

\makeatletter

\newcommand\restr[1]{\raisebox{-0.5ex}{\ensuremath{|_{#1}}}}
\newcommand{\rperp}{{\perp_\mathbb{R}}}

\newcommand{\Hardy}{{\mathbb H}}

\makeatother

\newcommand{\col}{\;:\;}

\newcommand{\Hil}{\mathcal{H}}
\newcommand{\Rl}{\mathbb{R}}
\newcommand{\Nl}{\mathbb{N}}
\newcommand{\Cl}{\mathbb{C}}
\newcommand{\Dl}{\mathbb{D}}

\newcommand{\Om}{\Omega}
\newcommand{\te}{\theta}
\newcommand{\la}{\lambda}
\newcommand{\ot}{\otimes}

\newcommand{\eps}{\varepsilon}
\newcommand{\Ran}{{\rm Ran}}
\newcommand{\Strip}{{\mathbb S}}

\def\Fix{\operatorname{Fix}}

\def\Ext{\operatorname{Ext}}
\def\ex{\operatorname{ex}}
\def\std{\operatorname{std}}

\def\supp{\operatorname{supp}}
\def\ker{\operatorname{ker}}

\def\dim{\operatorname{dim}}

\def\id{\operatorname{id}}

\DeclareMathOperator*{\wlim}{w-lim}

\def\Re{\operatorname{Re}}
\def\Im{\operatorname{Im}}

\def\id{\operatorname{id}}
\def\supp{\operatorname{supp}}
\def\Dom{\operatorname{Dom}}

\newcommand{\D}{\mathcal{D}}

\mathchardef\mhyphen="2D

 \newcommand{\CB}[0]{\mathcal{B}}
\newcommand{\CE}[0]{\mathcal{E}} \newcommand{\CF}[0]{\mathcal{F}}

\newcommand{\CK}[0]{\mathcal{K}} 
 
 \newcommand{\CP}[0]{\mathcal{P}}
\newcommand{\CQ}[0]{\mathcal{Q}} \newcommand{\CR}[0]{\mathcal{R}}
 
\newcommand{\CU}[0]{\mathcal{U}}

\newcommand{\N}{{\mathcal N}}
\newcommand{\B}{{\mathcal B}}

\makeatletter
\newcommand{\pushright}[1]{\ifmeasuring@#1\else\omit\hfill$\displaystyle#1$\fi\ignorespaces}
\newcommand{\pushleft}[1]{\ifmeasuring@#1\else\omit$\displaystyle#1$\hfill\fi\ignorespaces}
\makeatother

\usepackage[style=alphabetic, labelnumber, defernumbers=true, firstinits=true, maxbibnames=99, url=false, doi=true, isbn=false, eprint=false, backend=biber]{biblatex}
\renewbibmacro{in:}{}
\DeclareFieldFormat{pages}{#1}

\usepackage{xcolor}

\definecolor{green(munsell)}{rgb}{0.0, 0.66, 0.47}
\definecolor{BlueGreenn}{rgb}{0.3,0.5,0.8}
\definecolor{DB}{rgb}{0.3,0.3,0.3}
\definecolor{DOr}{rgb}{0.7,0.3,0.3}
\definecolor{DGr}{rgb}{0.3,0.7,0.3}
\definecolor{DBl}{rgb}{0.1,0.3,0.5}
\definecolor{arylideyellow}{rgb}{0.91, 0.84, 0.42}
\definecolor{burntorange}{rgb}{0.8, 0.33, 0.0}
\definecolor{chromeyellow}{rgb}{1.0, 0.65, 0.0}
\usepackage[hypertexnames=true, citecolor = burntorange, colorlinks = true, linkcolor = BlueGreenn, pdffitwindow = false, urlbordercolor = white,pdfborder={0 0 0}]{hyperref}%

\newbibmacro{string+doi}[1]{%
	\iffieldundef{doi}{#1}{\href{http://dx.doi.org/\thefield{doi}}{#1}}}
\DeclareFieldFormat{title}{\usebibmacro{string+doi}{\mkbibemph{#1}}}
\DeclareFieldFormat[article]{title}{\usebibmacro{string+doi}{\mkbibquote{#1}}}

\title{Inclusions of Standard Subspaces}

\author{Ricardo Correa da Silva}
\address{Department Mathematik, FAU Erlangen-Nürnberg, Germany;\linebreak Instituto de Física, Universidade São Paulo, Brazil\\ ricardo.correa.silva@fau.de}
\author{Gandalf Lechner}
\address{Department Mathematik, FAU Erlangen-Nürnberg, Germany\\ gandalf.lechner@fau.de}

\date{June 19, 2025}

\def\supp{\operatorname{supp}}
\def\ker{\operatorname{ker}}

\def\dim{\operatorname{dim}}

\def\id{\operatorname{id}}

\def\Re{\operatorname{Re}}
\def\Im{\operatorname{Im}}

\def\id{\operatorname{id}}
\def\supp{\operatorname{supp}}
\numberwithin{equation}{section}

\usepackage[style=alphabetic, labelnumber, defernumbers=true, firstinits=true, maxbibnames=99, url=false, doi=true, isbn=false, eprint=false, backend=biber]{biblatex}
\renewbibmacro{in:}{}
\DeclareFieldFormat{pages}{#1}

\begin{document}

\dedicatory{Dedicated to the memory of Huzihiro Araki}

\maketitle

\begin{abstract}
	Standard subspaces are closed real subspaces of a complex Hilbert space that appear naturally in Tomita-Takesaki modular theory and its applications to quantum field theory. In this article, inclusions of standard subspaces are studied independently of von Neumann algebras. Several new methods for their investigation are developed, related to polarizers, Gelfand triples defined by modular data, and extensions of modular operators. A particular class of examples that arises from the fundamental irreducible building block of a conformal field theory on the line is analyzed in detail.
\end{abstract}

\section{Introduction}

The modular theory of Tomita and Takesaki is one of the most powerful tools in the theory of von Neumann algebras and its applications (see, for example, \cite{Takesaki:2003,Blackadar:2006,Stratila:2020}). It has led to a finer classification of injective factors \cite{Connes:1976}, lies at the heart of thermal equilibrium states in quantum statistical mechanics \cite{BratteliRobinson:1997}, and has found many applications in quantum field theory \cite{Borchers:2000}.

Given a von Neumann algebra $\M\subset\B(\Hil)$ with a cyclic separating vector~$\Om$, the starting point of modular theory is Tomita's Theorem which states that the operators given by the polar decomposition of the closure $S$ of the involution $\M\Om\ni A\Om\to A^*\Om\in\M\Om$ define a unitary one-parameter group $\Delta^{it}$, $t\in\Rl$, acting by automorphisms on $\M$, and an antiunitary involution $J$ mapping $\M$ onto its commutant $\M'$.

The involution $S$ is fully described by its fixed point space
\begin{align}
	\Fix(S)
	=
	\{h\in\dom(S)\col Sh=h\}
	=
	\{A\Om\col A=A^*\in\M\}^{\|\cdot\|}
	,
\end{align}
which is an example of a {\em standard subspace}, namely a closed real linear subspace~$H$ of a complex Hilbert space $\Hil$ such that $H\cap iH=\{0\}$ and $H+iH$ is dense in~$\Hil$ (see Def.~\ref{def:standard-subspace}). It is apparent from this observation that modular theory can be considered entirely in terms of standard subspaces, regarding the von Neumann algebra $\M$ and the standard vector $\Om$ as additional data that are available in special cases. It has long been realized that this point of view is a natural setting for modular theory, simplifying and generalizing various aspects of the theory. Currently, Longo's lecture notes \cite{Longo:2008_2} are the standard reference for standard subspaces, to which we refer for more background information on the subject. See also Neeb's new lecture notes \cite{Neeb:2025} for a representation-theoretic perspective on standard subspaces.

\smallskip

The apparently earliest study of standard subspaces is due to Araki \cite{Araki:1963,Araki:1964}. Carried out before modular theory was formulated, he established an isomorphism between the lattice of standard subspaces in a Hilbert space~$\Hil$ and the Bose second quantized von Neumann algebras on the Fock over $\Hil$. After the advent of Tomita-Takesaki theory, the modular theory of standard subspaces was used to simplify Araki's proof \cite{EckmannOsterwalder:1973,LeylandsRobertsTestard:1978}, and the structure of second quantization factors was clarified in detail by Figliolini and Guido in \cite{FiglioliniGuido:1994}. Standard subspaces also played a prominent role in an approach to abstract modular theory put forward by Rieffel and van Daele \cite{RieffelvanDaele:1977}.

In more recent years, standard subspaces have been used for various other purposes. As a sample of the literature, we mention a few research directions: (i) Since any standard subspace can be defined in terms of a continuous unitary one-parameter group and an antiunitary involution commuting with it, one may construct (families of) standard subspaces from unitary group representations providing the one-parameter groups and involutions. In the context of positive energy representations of the Poincaré group, this idea is known as modular localization \cite{BrunettiGuidoLongo:2002}; it has recently been generalized significantly to other Lie groups \cite{MorinelliNeeb:2023,MorinelliNeeb:2021,Neeb:2025}. (ii) In QFT applications, standard subspaces serve as a means to encode localization regions. When combined with a twist to model interactions, they can be used to construct certain interacting QFTs \cite{CorreaDaSilvaLechner:2023,daSilvaGiorgettiLechner:2024}. (iii) Standard subspaces with translation symmetry are the subject of recent work in \cite{Schober:2024} and \cite{Koot:2025}, connecting to representation theory and spaces of holomorphic functions. (iv) In \cite{CiolliLongoRuzzi:2019}, a new concept of entropy, defined in terms of standard subspaces, is shown to capture the relative entropy of coherent states.

All these developments amply demonstrate that standard subspaces should be studied as a topic of interest in its own right, and not only as auxiliary objects appearing in the analysis of von Neumann algebras. Also when considered independently, standard subspaces retain many structural similarities to von Neumann algebra theory: While the structure of single standard subspaces $H$ is completely understood, the structure of an inclusion $K\subset H$ of two standard subspaces is much more involved. For instance, when equipped with the symplectic complementation given by the imaginary part of the scalar product, standard subspaces on a fixed Hilbert space form a non-orthomodular lattice, i.e. proper inclusions $K\subsetneq H$ may have trivial relative complement $K'\cap H=\{0\}$ \cite{FiglioliniGuido:1994} -- in close analogy to von Neumann subfactors $\N\subsetneq\M$ which may have trivial relative commutant $\N'\cap\M=\Cl1$, but in stark contrast to relative {\em orthogonal} complements of closed subspaces. As the relative commutant of an inclusion of von Neumann algebras is of central interest both for subfactors and quantum field theory, the analysis of the relative symplectic complement of inclusions $K\subset H$ of standard subspaces can also be seen as a testing ground for new ideas in these fields.

A particular type of inclusions are split inclusions \cite{DoplicherLongo:1984}, which are of central importance in quantum field theory, and can be translated to standard subspaces \cite{Morinelli:2018,LongoXu:2021}. The split property can be derived from modular nuclearity conditions and is a sufficient but by no means necessary condition for a large relative commutant. For applications in quantum field theory in higher dimensions, manageable weaker conditions would be desirable, and the standard subspace point of view can contribute here.

\medskip

In this paper, we initiate a general investigation of inclusions of standard subspaces, focussing mostly on understanding relative symplectic complements. Some of our results were already announced in \cite{daSilvaLechner:2023}.

In Section~\ref{section:standardsubspaces}, we review known material about standard subspaces that we rely on throughout the article. Section~\ref{section:inclusions} introduces inclusions $K\subset H$ of standard subspaces. We explain how the existence of singular inclusions is connected to unboundedness of the cutting projection, and characterize the relative complement of an inclusion $K\subset H$ in terms of the polarizer of $H$. We also translate modular compactness criteria and the split property to the standard subspace setting.

In Section~\ref{section:rigged} and Section~\ref{section:extensions}, we introduce two new methods for investigating inclusions of standard subspaces that are of independent interest. In Section~\ref{section:rigged}, we start from a standard subspace $H\subset\Hil$ and define a Gelfand triple (rigged Hilbert space) $\Hil_{H,+}\subset\Hil\subset\Hil_{H,-}$ in terms of $\Delta_H$. In typical examples, $\Hil$ is an $L^2$-space, whereas $\Hil_{H,+}$ consists of holomorphic functions, and $\Hil_{H,-}$ of distributions. It turns out that the different topologies provided by such a triple fit very well with the description of relevant data about inclusions $K\subset H$. In particular, we obtain a full understanding of inclusions that are defined by extending a given standard subspace $K$ by a closed real subspace $F$ to $K\subset H:=\overline{K+F}$.

Section~\ref{section:extensions} connects the study of inclusions $K\subset H$ of standard subspaces to extensions of the modular operator $\Delta_H$. Any such inclusion defines an extension $S_K^*S_H\supset\Delta_H$, and conversely, extensions with suitable properties define standard subspaces $K$ in $H$. While the operators $S_K^*S_H$ have subtle properties (for instance, they are not closable), we show that in case $K=UH\subset H$ with a unitary $U$, one may proceed to a kind of square root of $S_K^*S_H$, namely the closed operator $B_U:=U^*\Delta_H^{1/2}$, which is more amenable to operator-theoretic investigations. In particular, we show how the deficiency indices of $B_U$ can be used to give lower bounds on the dimension of the relative symplectic complement $(UH)'\cap H$. This analysis generalizes the bound state operators previously studied by Tanimoto \cite{Tanimoto:2015}.

In Section~\ref{section:standardpairsandinnerfunctions}, we study a family of examples of inclusions. Namely, we consider a standard subspace $H$ with half-sided translational symmetry, i.e. a unitary one-parameter group $T(x)$ with positive non-singular generator, such that $T(x)H\subset H$ for $x\geq0$ (a so-called {\em standard pair}, or {\em half-sided modular inclusion}). As this structure is closely tied to representations of the canonical commutation relations via Borchers' Theorem, it is unique up to unitary equivalence in the irreducible case, and represents the fundamental building block out of which more general standard pairs and corresponding conformal field theories can be constructed (see \cite{Wiesbrock:1993,GuidoLongoWiesbrock:1998,LechnerScotford:2022} for a few samples of a large body of literature).

For this example, all relevant operators such as the modular data, the real orthogonal projection onto $H$, and the polarizer can be computed as explicit integral operators.

Longo and Witten have shown that the semigroup of all unitary endomorphisms preserving $H$ and commuting with $T(x)$ can be identified with a semigroup of symmetric inner functions $\varphi$ on the upper half plane \cite{LongoWitten:2010}. Hence any such function $\varphi$ defines an inclusion $\varphi H\subset H$, which is split only in very special cases. In general, the properties of $\varphi H\subset H$ depend on the zeros and boundary behaviour of $\varphi$. We use a combination of the tools developed in this article and methods from approximation theory to determine the relative complement $\varphi H'\cap H$ in most cases and show that depending on the choice of $\varphi$, it can be $\{0\}$, finite-dimensional with any dimension $N\in\Nl$, infinite-dimensional but not cyclic, or cyclic.

\medskip

Our work is largely motivated from applications in quantum field theory such as in \cite{CorreaDaSilvaLechner:2023}. However, this article is dedicated to the mathematical structure of standard subspaces only, and applications to physics will appear elsewhere.

\section{Standard Subspaces and Their Modular Data}\label{section:standardsubspaces}

In this section we recall those structures canonically associated with a standard subspace that we will need, and introduce our notation. For a more detailed introduction to standard subspaces, see \cite{Longo:2008_2}. Proofs of all the claims made in this section can be found in that reference unless another reference is specified.

In the following, $\Hil$ always denotes a complex separable Hilbert space.

\begin{definition}\label{def:standard-subspace}
	A {\em standard subspace} is a closed $\Rl$-linear subspace $H\subset\Hil$ that is cyclic (meaning $H+iH\subset\Hil$ is dense) and separating (meaning $H\cap iH=\{0\}$). The set of all standard subspaces of $\Hil$ is denoted $\Std(\Hil)$.
\end{definition}

\subsection{Modular Data}

The set $\Std(\Hil)$ is in bijection with the family of all closed densely defined antilinear involutions on $\Hil$ via mapping $H\in\Std(\Hil)$ to its {\em Tomita operator}
\begin{align}
	S_H:H+iH\to H+iH,\qquad h_1+ih_2\mapsto h_1-ih_2,
\end{align}
and recovering $H$ from $S_H$ as\footnote{The notation $\Fix(A)$ for the fixed points of a map $A:D\to D$ will be used throughout.} $H=\ker(S_H-1)=:\Fix(S_H)$. As usual, the polar decomposition of $S_H$ is written as $S_H=J_H\Delta_H^{\nicefrac{1}{2}}=\Delta_H^{-\nicefrac{1}{2}}J_H$. In particular,
\begin{align}
	\dom S_H=\dom\Delta_H^{\nicefrac{1}{2}}=H+iH=:\Cl[H],
\end{align}
where we have introduced the notation $\Cl[X]=X+iX$ for the complex linear span of real subspace $X\subset\Hil$.

The modular operator $\Delta_H$ is positive and non-singular, and induces automorphisms of $H$, namely $\Delta_H^{it}H=H$, $t\in\Rl$. The modular conjugation $J_H$ maps $H$ onto its symplectic complement
\begin{align}
	J_HH=H':=\{v\in\Hil\col\Im\langle v,h\rangle=0\;\forall h\in H\}.
\end{align}
Also $H'$ is a standard subspace, and $(H')'=H$ as well as $J_{H'}=J_H$, $\Delta_{H'}=\Delta_H^{-1}$ hold. Borrowing terminology from von Neumann algebras as usual, we will say that $H$ is a {\em factor} if $H\cap H'=\{0\}$; , this is equivalent to $1$ not being an eigenvalue of $\Delta_H$ \cite{FiglioliniGuido:1994}. We will say that $H$ is {\em abelian} if $H\subset H'$ (in this case, it is {\em maximally abelian}, namely $H=H'$).

\subsection{Projections and Polarizers}

As for closed complex subspaces, we will often study standard subspaces through projections. That is, we consider $\Hil$ as a real Hilbert space with scalar product $\Re\langle\,\cdot\,,\,\cdot\,\rangle$, and the real orthogonal complement
\begin{align}
	H^{\perp_\Rl}=iH'.
\end{align}
Given $H\in\Std(\Hil)$, we will write $E_H$ for the bounded real linear real orthogonal projection with range $H$. We have \cite{FiglioliniGuido:1994}
\begin{align}
	E_H = (1+S_H)(1+\Delta_H)^{-1}.
\end{align}
In case $H$ is a factor, we also consider the projection
\begin{align}
	P_H:H+H'\to H+H', \qquad h+h'\mapsto h.
\end{align}
This is a densely defined closed real linear projection with domain $H+H'$ and real adjoint $P_H^*=P_{iH}=-iP_H i$ \cite{CiolliLongoRuzzi:2019}, called the {\em cutting projection}. It can be expressed as the operator closure
\begin{align*}
	P_H = \overline{(1+S_H)(1-\Delta_H)^{-1}},
\end{align*}
from which it is easy to see that the cutting projection is bounded if and only if $H+H'=\Hil$ which holds if and only if $1\not\in\sigma(\Delta_H)$.

The {\em polarizer} \cite{Longo:2022_2} $D_H$ of $H$ is defined as the unique bounded real linear operator $D_H:H\to H$ satisfying $\Im\langle h,k\rangle=\Re\langle h,D_Hk\rangle$ for all $h,k\in H$. It satisfies $\|D_H\|\leq1$, $D_H^*=-D_H$ (real adjoint), and
\begin{align}\label{eq:polarizer}
	D_H &= -E_H i |_H = i\frac{\Delta_H-1}{\Delta_H+1}\bigg|_H.
\end{align}
We have $\ker D_H=H\cap H'$, so the polarizer is invertible if and only if $H$ is a factor. In that case, its inverse is given by \cite{Longo:2022_2}
\begin{align*}
	D_H^{-1}
	&=
	P_Hi|_H
	=
	-i\frac{\Delta_H+1}{\Delta_H-1}\bigg|_H,
\end{align*}
which is bounded if and only if $P_H$ is bounded.

\subsection{Second Quantization Factors and Types} The original motivation to study standard subspaces derives from the fact that they describe the local algebras of a quantum field theory of free Bosons. Namely, denoting by $W(h)$, $h\in\Hil$, the Weyl operators on the Bosonic Fock space over $\Hil$, the von Neumann algebra $\CR(H):=\{W(h)\col h\in H\}''$ has the Fock vacuum $\Om$ as a standard vector if and only if $H\subset\Hil$ is standard. The map
\begin{align}
	H \mapsto \CR(H)
\end{align}
from closed real subspaces of $\Hil$ to von Neumann algebras acting on the Fock space over $\Hil$ is an isomorphism of complemented lattices, and the modular data of $(\CR(H),\Om)$ are given by second quantization of the modular data of $H$ \cite{EckmannOsterwalder:1973,LeylandsRobertsTestard:1978}.

This connection suggests to import more terminology from the von Neumann algebraic picture into the standard subspace picture. In particular, we define the {\em type} of a standard subspace $H\in\Std(\Hil)$ as the Connes-Murray-von Neumann type of $\CR(H)$. The relevant information on types is summarized in the following theorem.

\begin{theorem}\label{prop:typeI}
	Let $H\in\Std(\Hil)$.
	\begin{enumerate}
		\item $H$ is either of type I or III.
		\item $H$ is of type I if and only if the restriction of its modular operator $\Delta_H$ to its spectral subspace corresponding to the interval $[0,1]$ is of trace class.
		\item If $H$ is a factor with $1\in\sigma(\Delta_H)$ ($P_H$ is unbounded), then $H$ is of type~III${}_1$.
	\end{enumerate}
\end{theorem}
a) was proven before modular theory by Araki in \cite{Araki:1964_2}, see \cite[Thm.~1.3]{FiglioliniGuido:1994} for the following standard subspace formulation, where also b) and c) are proven.  We refer to \cite{Araki:1963,LongoXu:2021,Longo:2022_2} for various alternative characterizations of the type I property in terms of $\Delta_H$, $E_H$, or $D_H$.

In the following, we will say that a modular operator $\Delta_H$ has a {\em spectral gap} if $1\not\in\sigma(\Delta_H)$. In this case, $H$ is a factor and $D_H$ is invertible with domain~$H$. 		This situation occurs in particular when $H$ is a type I factor.

We shall say that an inclusion $K\subset H$ of standard subspaces is {\em split} if there exists a type I factor standard subspace $L$ such that $K\subset L\subset H$ \cite{DoplicherLongo:1984}. We shall say that an inclusion $K\subset H$ of standard subspaces is {\em standard split} if it is split and $K'\cap H$ is cyclic (hence standard). This definition can also be found in \cite{LongoXu:2021,Morinelli:2018}.

\section{Inclusions and Relative Symplectic Complements}\label{section:inclusions}

Our main subject of interest in this article are inclusions of standard subspaces. Given $H\in\Std(\Hil)$, we write
\begin{align}
	\Std(H)
	:=\{K\in\Std(\Hil) \col K\subset H\}
\end{align}
for the set of standard subspaces included in $H$. We first recall that this is an interesting structure only for unbounded $\Delta_H$, and in particular only in infinite-dimensional Hilbert spaces. The following result is known in various formulations and variations (see, e.g. \cite{FiglioliniGuido:2000}).

\begin{lemma}\label{lemma:extendK}
Let $H\in\Std(\Hil)$. The following are equivalent:
\begin{enumerate}
	\item There exists $K\in\Std(\Hil)$ with $H\subsetneq K$.
	\item There exists $K\in\Std(\Hil)$ with $K\subsetneq H$.
	\item $\Delta_H$ is unbounded.
\end{enumerate}
\end{lemma}
\begin{proof}
	$a)\Rightarrow c)$ For an inclusion $H\subset K$, we have an extension $S_H\subset S_K$ of Tomita operators. If $\Delta_H$ is bounded, $S_H$ is bounded, and hence $H+iH=\Hil$. Then $S_K=S_H$ on all of $\Hil$, and therefore $K=H$.

	$c)\Rightarrow a)$ Since $H+iH\neq\Hil$, there exists a finite dimensional real linear separating subspace $F\subset\Hil$ with $(F+iF)\cap(H+iH)=\{0\}$, for instance $F=\Rl v$ for any vector $v\not\in H+iH$.

	We define $K:=H+F$, which is a closed real linear subspace containing $H$, hence cyclic. To show it is separating, let $k=h+f\in H+F$ coincides with $i(\tilde h+\tilde f)$ (with $\tilde h\in H$, $\tilde f\in F$), then $i\tilde f-f=h-i\tilde h\in(H+iH)\cap(F+iF)=\{0\}$, and thus $h=f=0$ because both $H$ and $F$ are separating.

	The statement $b)\Leftrightarrow a)$ follows by taking symplectic complements and observing that $\Delta_H$ is unbounded if and only if $\Delta_{H'}=\Delta_H^{-1}$ is unbounded.
\end{proof}

Inclusions $\N\subset\M$ of von Neumann algebras with a joint cyclic separating vector define inclusions $K\subset H$ of the associated standard subspaces, and the inclusion $K\subset H$ shares some similarities with inclusions of von Neumann algebras such as subfactors. For instance, just as von Neumann algebra inclusions $\N\subset\M$ can be singular (irreducible) in the sense $\N'\cap\M=\Cl1$, standard subspace inclusions $K\subset H$ can be singular in the sense $K'\cap H=\{0\}$, i.e. $K$ has trivial relative symplectic complement in $H$.

Obviously $K'\cap H$ contains the centers $K'\cap K$ and $H'\cap H$ of $K$ and $H$, so for singular inclusions it is necessary that $K$ and $H$ are factors. The other deciding factor is the cutting projection, as we show now (see \cite{FiglioliniGuido:2000} for a different formulation).

\begin{proposition}\label{prop:singular-extension}
	Let $K\in\Std(\Hil)$. The following are equivalent:
	\begin{enumerate}
		\item There exists $H\in\Std(\Hil)$ with $K\subsetneq H$ and $K'\cap H=\{0\}$.
		\item There exists $H\in\Std(\Hil)$ with $H\subsetneq K$ and $H'\cap K=\{0\}$.
		\item $K$ and $H$ are factors and the four operators $\Delta_K, \Delta_H, P_K,P_H$ are all unbounded.
	\end{enumerate}
\end{proposition}
\begin{proof}
	$a)\Rightarrow c)$ The assumption implies directly that $K$ and $H$ are factors with unbounded modular operators. Suppose $P_K$ is bounded. Then $K+K'=\Hil$ and any $h\in H$ can be decomposed as $h=k+k'$, i.e. $k'=h-k\in K'\cap H\Rightarrow h=k$. Thus $H\subset K$, i.e. $H=K$, a contradiction to our assumption $K\subsetneq H$. The assumption $\|P_H\|<\infty$ leads to a contradiction in an analogous manner.

	$c)\Rightarrow a)$ We claim that $c)$ implies $(K+iK)\cup(K+K')\neq\Hil$. In fact, the union $(K+iK)\cup(K+K')$ is a linear space only if one of the two domains $K+iK$ and $K+K'$ contains the other, so $(K+iK)\cup(K+K')=\Hil$ implies that $K+iK=\Hil$, a contradiction to $\Delta_K^{\nicefrac{1}{2}}$ being unbounded, or $K+K'=\Hil$, namely $1\not\in\sigma(\Delta_K)$, a contradiction to $P_K$ being unbounded.

	Pick a vector $v\not\in(K+iK)\cup(K+K')$, and consider $H:=K+\Rl v$. Then $H$ is closed and cyclic. Since we have $\Cl v\cap(K+iK)=\{0\}$, we see that $H$ is also separating as in Lemma~\ref{lemma:extendK}.

	We claim that the inclusion $K\subset H$ is singular (it is proper by construction). In fact, the elements of the relative symplectic complement are of the form $h=k+\la\,v=k'\in K'$ for some $\la\in\Rl$, i.e. $\la\,v=k'-k$. But we have $v\not\in(K+K')$ by choice of $v$, so $\la=0$ and $k=k'$, which implies $k=k'=0$ because $K$ is a factor.

	$a)\Leftrightarrow b)$ follows by taking symplectic complements and observing that statement $c)$ does not change when replacing $K$ by $K'$.
\end{proof}

\begin{corollary}\label{cor:singularincltypeIII}
	If $K\subsetneq H$ is a singular inclusion, i.e. $K'\cap H=\{0\}$, then $K$ and $H$ are factors of type III${}_1$.
\end{corollary}
\begin{proof}
	By Prop.~\ref{prop:singular-extension}, we know that $1$ is not an eigenvalue of $\Delta_K$ or $\Delta_H$ (the spaces are factors), but $1$ lies in the spectrum of these operators because the cutting projections are unbounded. Hence $1\in\sigma_{\rm ess}(\Delta_K)\cap\sigma_{\rm ess}(\Delta_H)$, implies the III${}_1$ property by Thm.~\ref{prop:typeI}.
\end{proof}

\subsection{Real Orthogonal Complements and the Polarizer}\label{subsec:Realorthogonalpolarizer}

Whereas the relative symplectic complement $K'\cap H$ of a proper inclusion $K\subsetneq H$ may be trivial, the relative real orthogonal complement $K^{\perp_\Rl}\cap H=H\ominus_\Rl K\neq\{0\}$ never is. The polarizer $D_H$ \eqref{eq:polarizer} mediates between the real and imaginary part of the scalar product and is therefore useful for comparing these two types of relative complements, and we now state a lemma gathering relevant properties of $D_H$. However, as $D_H$ is invertible only in case $H$ is a factor, we need to split off the center of $H$ in general, i.e. consider the spectral projection $Q_H$ of $\Delta_H$ for eigenvalue $1$, namely $\Ran Q_H=\Cl[H\cap H']$, and the decomposition
\begin{align}
	\Hil=Q_H\Hil\oplus Q_H^\perp\Hil,\qquad
	H=(H\cap H')\oplus \tilde H,\quad \tilde H:=Q_H^\perp H.
\end{align}
Note that $\tilde H\in\Std(Q_H^\perp\Hil)$.

\begin{lemma}\label{lemma:relativesymp}
	Let $K\subset H$ be an inclusion of standard subspaces and $D_H$ the polarizer of $H$.
	\begin{enumerate}
		\item\label{item:splittingHbyK} We have the real orthogonal direct sum decomposition
		\begin{align}\label{eq:H-polarizer-split}
			H=(K'\cap H) \oplus_\Rl \overline{D_HK},
			\qquad
			K'\cap H=\ker(E_KD_H).
		\end{align}

		\item\label{item:splitoffcenter} When splitting $H$ into its center and its abelian part $\tilde H$, we have
		\begin{align}
			K'\cap H
			&=
			(H'\cap H)\oplus
			D_{\tilde H}^{-1}(K^{\perp_\Rl}\cap {\tilde H}\cap\dom D_{\tilde H}^{-1})
			.
		\end{align}

		\item\label{item:cyclicityandpolarizer} If $K'\cap H$ is cyclic, then
		\begin{align}
			(K'\cap H)'\cap\overline{D_HK}=\{0\}.
		\end{align}

		\item\label{item:KcapDK} $K\cap D_HK\cap \dom\Delta_H\subset K\cap(iK+iH')\subset K\cap D_HK$
	\end{enumerate}
\end{lemma}
\begin{proof}
	\ref{item:splittingHbyK} Let $h\in H$. Then $h\in K'$ if and only if for all $k\in K$, we have $\Im\langle k,h\rangle=\Im\langle E_Kk,h\rangle=0$. But
	\begin{align*}
		\Im\langle E_Kk,h\rangle
		=
		\Re\langle E_Kk,D_Hh\rangle
		=
		\Re\langle k,E_KD_Hh\rangle.
	\end{align*}
	This vanishes for all $k\in K$ if and only if $E_KD_Hh=0$, proving $K'\cap H=\ker(E_KD_H)$. For the direct sum decomposition, we have
	\begin{align*}
		\ker(E_KD_H)^\rperp\cap H
		=
		\overline{\Ran((E_KD_H)^*)}
		=
		\overline{\Ran(D_HE_K)},
	\end{align*}
	where we have used $E_K^*=E_K$ and $D_H^*=-D_H$.

	\ref{item:splitoffcenter} In $Q_H\Hil$, we clearly have $Q_H(K'\cap H)=H'\cap H$ and in the orthogonal complement, the polarizer $D_{\tilde H}$ is invertible. Since $E_KD_{\tilde H}h=0$ is equivalent to $D_{\tilde H}h\in K^{\perp_\Rl}\cap {\tilde H}\cap\dom D_{\tilde H}^{-1}$, the claimed equality follows.

	\ref{item:cyclicityandpolarizer} According to \ref{item:splittingHbyK}, the real orthogonal complement of $(K'\cap H)$ in $\Hil$ is $(K'\cap H)^\rperp=\overline{D_HK}\oplus_\Rl H^\rperp$. Hence
	\begin{align*}
		(\Cl[K'\cap H])'
		=
		(K'\cap H)'\cap (K'\cap H)^\rperp
		=
		(K'\cap H)'\cap (\overline{D_HK}\oplus_\Rl H^\rperp),
	\end{align*}
	which is $\{0\}$ if and only if $K'\cap H$ is cyclic. In particular, cyclicity of $K'\cap H$ implies $(K'\cap H)'\cap \overline{D_HK}=\{0\}$.

	\ref{item:KcapDK} Let $k=D_H\tilde k\in K\cap\dom\Delta_H$, with $k\in K$. Then $k=i(\Delta_H+1)^{-1}(\Delta_H-1)\tilde k$, so
	\begin{align*}
		(\Delta_H+1)k=i(\Delta_H-1)\tilde k
		\Rightarrow
		-(k+i\tilde k)
		=
		\Delta_H(k-i\tilde k)
		=
		S_{H'}(k+i\tilde k),
	\end{align*}
	which implies $k+i\tilde k\in iH'$, i.e. $k\in (iH'+iK)$.

	For the second inclusion, let $k\in K\cap (iH'+iK)$, i.e. there exists $\tilde k\in K$ such that $k+i\tilde k\in iH'=H^\rperp$. Therefore $0=E_H(k+i\tilde k)=k-D_H\tilde k$, namely $k=D_H\tilde k\in D_HK$.
\end{proof}

\pagebreak
\begin{remark}
	Let us emphasize some immediate consequences of this lemma:
	\begin{enumerate}
		\item If $H$ is a factor, the relative real orthogonal complement of the inclusion $K\subset H$ is larger than the relative symplectic complement in the sense of $\dim(K^{\perp_\Rl}\cap H)\geq\dim(K'\cap H)$. So a large relative real orthogonal complement is necessary for a large relative symplectic complement for inclusions in a factor.

		\item When $\Delta_H$ has a spectral gap, $D_H$ yields an isomorphism between $K'\cap H$ and $K^\rperp \cap H$, and $K\subset H$ is singular if and only if $K=H$. This situation occurs in particular when $H$ is a type I factor and was already used in Corollary~\ref{cor:singularincltypeIII}.
	\end{enumerate}
\end{remark}

The orthogonal split \eqref{eq:H-polarizer-split} given by the polarizer of $H$ becomes most effective when $D_HK$ is closed, as we now discuss in a corollary. Later we will see that closedness of $D_HK$ follows from a modular compactness condition.

\begin{corollary}\label{cor:closedrange}
	Let $K\subset H$ be an inclusion of standard subspaces, with $H$ a factor, and assume that $D_HK$ is closed. Then
	\begin{align}\label{KHDspace}
		(K'\cap H)'\cap D_H K
		=
		K\cap (D_HK),
	\end{align}
	and if $K'\cap H$ is cyclic,
	\begin{align}
		K\cap (D_HK)
		=
		H'\cap\Cl[K]=\{0\}.
	\end{align}
\end{corollary}
\begin{proof}
	The inclusion ``$\supset$'' in \eqref{KHDspace} follows immediately from $(K'\cap H)'\supset K$. To establish the inclusion ``$\subset$'', let $h=D_Hk$, with $k\in K$, such that $h\in(K'\cap H)'$. Then for any $\tilde h\in K'\cap H$ we have
	\begin{align}
		0
		=
		\Im\langle D_Hk,\tilde h\rangle
		=
		\Re\langle D_Hk,D_H\tilde h\rangle
		=
		-\Re\langle D_H^2k,\tilde h\rangle,
	\end{align}
	i.e. $D_H^2k\in(K'\cap H)^\rperp\cap H=D_HK$ (here the closedness of $D_HK$ is used). As $H$ is assumed to be a factor, the polarizer is injective, and we conclude $D_Hk\in K$. This finishes the proof of \eqref{KHDspace}.

	The consequence of cyclicity of $K'\cap H$ is immediate from Lemma~\ref{lemma:relativesymp}~\ref{item:cyclicityandpolarizer}. For the last statement, we already know $K\cap D_HK\supset K\cap(iK+iH')$ from Lemma~\ref{lemma:relativesymp}~\ref{item:KcapDK}. The equivalence of $K\cap(iK+iH')=\{0\}$ and $H'\cap \Cl K=\{0\}$ can be checked by standard arguments.
\end{proof}

\subsection{The Endomorphism Semigroup}
In applications, inclusions of standard subspaces in $H\in\Std(\Hil)$ are often given by unitaries that act as endomorphisms on $H$. We first remark that this is not the general case by using a variation of the argument in the proof of part $c)\Rightarrow a)$ of Prop.~\ref{prop:singular-extension}.

\begin{remark}
	Let $K\in\Std(\Hil)$ be a factor with unbounded $\Delta_K$. Then there exists $v\in K'$, $v\not\in K+iK$ (otherwise we would have $\dom\Delta_K^{-\nicefrac{1}{2}}=K'+iK'\subset K+iK=\dom\Delta_K^{\nicefrac{1}{2}}$, which implies $\|\Delta_K\|<\infty$). Defining $H:=K+\Rl v$, we then have $v\in H'=K'\cap\{v\}'$, i.e. $v$ lies in the center of $H$ and we have $K'\cap H=\Rl v$ and $\Delta_H v=v$. But if we had $K=UH$ for some unitary $U$, then also $\Delta_H$ would not have eigenvalue $1$.
\end{remark}

We define the {\em endomorphism semigroup of $H\in\Std(\Hil)$} as
\begin{align}
	\CE(H)
	:=
	\{U\in\CU(\Hil)\col UH\subset H\}.
\end{align}
Recall that $U\in\CE(H)$ is equivalent to $US_H\subset S_HU$, and also equivalent to $F_U(t):=\Delta_H^{it}U\Delta_H^{-it}$ extending to a bounded strongly continuous map on the closed strip $\overline{\Strip_{-\nicefrac{1}{2}}}$ that is analytic in the interior and satisfies $F_U(-\frac{i}{2})=J_HUJ_H$ \cite[Thm.~2.12]{ArakiZsido:2005}.

Clearly, $U\in\CE(H)$ is equivalent to $U^*\in\CE(H')$ and, in particular, selfadjoint endomorphisms are automorphisms, namely $UH=H$. From the point of view of inclusions, automorphisms are trivial. It will turn out to be more interesting to look at the fixed points in $\CE(H)$ under the involution
\begin{align}\label{eq:Uhash}
	U^\#:=J_HU^*J_H,
\end{align}
studied also in \cite{Neeb:2021_3}. We call $U\in\CE(H)$ {\em symmetric}\footnote{With any $U\in\CE(H)$, the product $U\cdot J_HU^*J_H$ is a symmetric endomorphism. Given an inclusion $K\subset H$, the {\em canonical endomorphism} $U:=J_KJ_H$ is a symmetric endomorphism.} if $U=U^\#$. Symmetric endomorphisms will be important in Section~\ref{section:extensions}.

\medskip

For later reference, we also record two results on iterated endomorphisms and inclusions, respectively.

\begin{proposition}\label{prop:limitendo}
	Let $(U_n)_{n\in \Nl}\subset \CE(H)$ be such that
	\begin{enumerate}
		\item $U_{n+1} H \subset U_n H$ for all $n\in \Nl$, and
		\item $U:=\wlim\limits_{n\to\infty} U_n$ exists and lies in $\CE(H)$ (is unitary).
	\end{enumerate}
	Then
	\begin{equation}\label{eq:limitendos}
		\bigcap_{n\in \Nl}U_n H = U H,  \qquad \qquad \overline{\bigcup_{n\in \Nl}U_n H'} = U H',
	\end{equation}
	and
	\begin{equation}\label{eq:limitendorelativecomplement}
		\overline{\left(\bigcup_{n\in \Nl}U_n H'\right)\cap H} = U H'\cap H.
	\end{equation}
\end{proposition}
\begin{proof}
	Noting that $H$ is, as any closed subspace of a Hilbert space, weakly closed, we observe that assumption a) implies that for $h\in H$ and $n,k\in \Nl$, $U h=\wlim_{k\to\infty}U_{n+k}h\in U_n H$. Hence we have $U H\subset \bigcap_{n\in \Nl}U_n H$ and we need to show the reverse inclusion. Let $k\in\bigcap_{n\in \Nl}U_n H$, then there exists $h_n\in H$ such that $k=U_n h_n$ for every $n\in \Nl$. Thus, $h_n=U_n^\ast k$ and $\wlim_{n\to \infty} h_n= U^\ast k$.

	We conclude that $U^\ast k=h\in H$, which is equivalent to $k=U h\in U H$. The second equation in \eqref{eq:limitendos} follows simply by taking the symplectic complement. Similarly, \eqref{eq:limitendorelativecomplement} is equivalent to
	$\overline{\left(\bigcap_{n\in \Nl}U_n H\right)+ H'} = \overline{U H+ H'}$,
	by taking symplectic complements, and application of \eqref{eq:limitendos} yields \eqref{eq:limitendorelativecomplement}.
\end{proof}

Another simple statement about iterated inclusions that we will use later is the following.

\begin{lemma}\label{lemma:iteratedinclusions}
	Let $K\subset L\subset H$ be standard subspaces.
	\begin{enumerate}
		\item\label{item:realsplit} $K^\rperp\cap H=(K^\rperp\cap L)\oplus_\Rl (L^\rperp\cap H)$ as a real orthogonal direct sum.
		\item\label{item:symplecticsplit} If $L$ is a factor, $(K'\cap L)\oplus(L'\cap H)\subset K'\cap H$ as an algebraic direct sum (no orthogonality).
	\end{enumerate}
\end{lemma}
\begin{proof}
	\ref{item:realsplit} is immediate by orthogonal decomposition. For \ref{item:symplecticsplit}, we note that the two spaces $K'\cap L$ and $L'\cap H$ are both contained in $K'\cap H$, and their intersection $K'\cap L\cap L'\cap H$ is $\{0\}$ in case $L$ is a factor.
\end{proof}

We remark that in \ref{item:symplecticsplit}, in general one does not have $(K'\cap L)\oplus(L'\cap H)=K'\cap H$. An example will be given in Remark~\ref{remark:finiteblaschke}.

\subsection{Modular Compactness Conditions and Split Inclusions}

As in the von Neumann algebraic situation, we will say that an inclusion $K\subset H$ of standard subspaces satisfies {\em modular compactness} if the bounded operator $\Delta_H^{1/4}E_K$ is compact, and {\em modular nuclearity} if $\Delta_H^{1/4}E_K$ is trace class. In this section we establish standard subspace analogues of the relation of such conditions to split inclusions \cite{BuchholzDAntoniLongo:1990-1}.

Our first result derives the closed range assumption on $D_HK$ (Cor.~\ref{cor:closedrange}) from modular compactness.

\begin{proposition}
	Let $K\subset H$ be an inclusion of standard subspaces satisfying modular compactness. Then $D_HK$ is closed.
\end{proposition}
\begin{proof}
	We have
	\begin{align*}
		D_HE_K
		=
		i(\Delta_H-1)(\Delta_H+1)^{-1}E_K
		=
		-i\left(1-2\Delta_H^{3/4}(\Delta_H+1)^{-1}\cdot \Delta_H^{1/4}E_K\right)E_K.
	\end{align*}
	Since $C:=\frac{2\Delta_H^{3/4}}{1+\Delta_H}\cdot \Delta_H^{1/4}E_K$ is compact, $1-C$ is Fredholm and therefore has closed range. This implies that $D_HK$ is closed.
\end{proof}

As in the von Neumann algebra situation, the split property implies modular compactness, but with a simpler argument.

\begin{lemma}
	\leavevmode
	\begin{enumerate}
		\item If $L$ is a type I standard subspace, then $|\Delta_L^{1/4} E_L|^4$ is trace class, and in particular, $\Delta_L^{1/4} E_L$ is compact.

		\item Let $K\subset H$ be split. Then $|\Delta_H^{1/4} E_K|^4$ is trace class, and in particular, modular compactness holds.
	\end{enumerate}
\end{lemma}
\begin{proof}
	a) Consider the operator $X:=\Delta_L^{1/4}(1+\Delta_L)^{-1}$ and split $\Hil$ into a direct sum $\Hil_<\oplus\Hil_1\oplus\Hil_>$ corresponding to spectrum of $\Delta_L$ in $(0,1)$, $\{1\}$, and $(1,\infty)$. Recall that $\Delta_L|_{\Hil_<\oplus\Hil_1}$ is trace class (Thm.~\ref{prop:typeI}~b)). Hence $X^4|_{\Hil_<\oplus\Hil_1}$ is trace class. Since $\Hil_>=J_L\Hil_<$ and $J_LXJ_L=\Delta_L^{3/4}(1+\Delta_L)^{-1}$, we see that $X^4|_{\Hil_>}$ is also trace class.

	The operator under consideration is
	\begin{align*}
		\Delta_L^{1/4} E_L
		=
		(1+J_L)\Delta_L^{1/4}(1+\Delta_L)^{-1}
		,
	\end{align*}
	which implies the claim.

	b) Let $L$ be an interpolating type I space, $K\subset L\subset H$. Since
	\begin{align*}
		\Delta_H^{1/4} E_K
		=
		\Delta_H^{1/4}\Delta_L^{-{1/4}}\cdot \Delta_L^{1/4} E_L\cdot E_K,
	\end{align*}
	and the first and third factors are bounded, the claim follows from part a).
\end{proof}

\begin{proposition}
	Let $K\subset H$ be an inclusion of standard type III factor subspaces.
	\begin{enumerate}
		\item If $K\subset H$ satisfies modular nuclearity, then $K\subset H$ is split.
		\item If $K\subset H$ is split, then $K'\cap H\in\Std(\overline{\Cl[K'\cap H]})$ is of type III. In particular, $\dim(K'\cap H)=\infty$.
	\end{enumerate}
\end{proposition}
\begin{proof}
	a) We give a very indirect proof by considering the (Bose) second quantized von Neumann algebras $\CR(K)\subset\CR(H)$ on the Fock space over $\Hil$. Taking into account that the Fock vacuum is standard for these algebras and their modular operators arise from those of the standard subspaces by second quantization \cite{EckmannOsterwalder:1973}, we may apply \cite[Lemma 2.3, Thm.~3.3, Lemma~3.8]{LechnerSanders:2016} (note that we have $\|\Delta_H^{1/4}E_K\|<1$) to conclude that the inclusion $\CR(K)\subset\CR(H)$ satisfies the modular nuclearity condition \cite{BuchholzDAntoniLongo:1990-1}. Hence it is split, which implies that $K\subset H$ is split \cite[Prop.~2.14]{LongoXu:2021}.

	b) As $K$ and $H$ are of type III, $\CR(K)\subset\CR(H)$ is a split inclusion of type III factors. Hence there is a spatial isomorphism $\CR(K)\vee\CR(H)'\cong\CR(K)\ot\CR(H)'$ \cite[Thm.~1]{DAntoniLongo:1983}, and we conclude that $\CR(K'\cap H)\cong\CR(K')\ot\CR(H)$ is of type III.

	As $K'\cap H$ might not be cyclic in $\Hil$, we view it as a standard subspace in the auxiliary Hilbert space $\CK:=\overline{\Cl[K'\cap H]}$. Then $K'\cap H$ is a type III standard subspace by the above argument. Since finite-dimensional standard subspaces are type I (Thm.~\ref{prop:typeI}~b)), we also conclude $\dim(K'\cap H)=\infty$.
\end{proof}

It would be interesting to find a more direct standard subspace type argument for this result.

We thus see, as in the von Neumann algebra situation \cite{BuchholzLechner:2004}, that modular nuclearity implies large relative symplectic complements. When combined with additional spectral information on the inclusion, one can also establish cyclicity of $K'\cap H$ \cite{Lechner:2008}. However, as is well known, modular nuclearity or split is certainly not necessary for cyclicity of $K'\cap H$. We will therefore develop independent techniques for the investigation of $K'\cap H$ in the following sections.

\medskip

We conclude this section by a sequence characterization of modular compactness that we will use later.

\begin{proposition} \label{prop:notcompactEK}
	Let $K\subset H$ be standard subspaces. Then modular compactness holds if and only if for any two sequences of unit vectors $(h_n')_n\subset H'$, $(k_n)_n\subset K$ that converge weakly to zero, one also has $\lim_n\langle h_n',k_n\rangle=0$.
\end{proposition}
\begin{proof}
	If $\Delta_H^{1/4}E_K$ is not compact, there exists a sequence $(\tilde{k}_n)_n \subset K$ of unit vectors such that $\Delta_H^\frac{1}{4}\tilde{k}_n$ has no convergent subsequence.  Notice that $\tilde k_n$ does not have a convergent subsequence because $\Delta_H^{\frac{1}{4}}\restr{K}$ is bounded, but, since $(\tilde{k}_n)_n$ is bounded, it has a {\em weakly} convergent subsequence, which from now on will denoted by $(\tilde{k}_n)_n$. Since $K$ is weakly closed, we have $\tilde k:=\wlim_{n\to\infty} \tilde{k}_n\in K$. Therefore, $k_n:=\frac{(\tilde{k}_n-\tilde{k})}{\|\tilde{k}_n-\tilde{k}\|}\in K$ defines a sequence which converges weakly to zero and does not have a convergent subsequence. Also $\Delta_H^{1/4}k_n$ does not have a convergent subsequence by the same argument as above.

	Finally set $h^\prime_n=J_Hk_n\in H^\prime$ which is also normalized and weakly converging to zero. Then,
	\begin{align*}
		\ip{k_n}{h^\prime_n}=\ip{k_n}{J_Hk_n}=\langle k_n,\Delta_H^{1/2} k_n\rangle=\|\Delta_H^\frac{1}{4}k_n\|^2\not\to 0.
	\end{align*}

	Conversely, let $(h_n')_n$ and $(k_n)_n$ be as in the statement. By taking a subsequence if necessary, there exists $\eps>0$ such that
	\begin{align}\label{eq:awayfromzero}
		|\langle h_n',k_n\rangle|>\eps,\qquad n\in\Nl.
	\end{align}
	If $\Delta_H^{1/4}E_K$ was compact, we would have a convergent subsequence
	\begin{equation}\label{eq:limeq}
	\lim_{n\to\infty}\Delta_H^\frac{1}{4} k_{n_j}=\wlim_{n\to \infty} \Delta_H^\frac{1}{4} k_{n_j}=0.
	\end{equation}

	By setting $h_n:=J_H h_n^\prime$, which also converges weakly to zero, we would have
	\begin{align*}
		\lim_{n\to \infty}|\ip{k_n}{h^\prime_n}|&=\lim_{n\to \infty}|\ip{k_n}{J_Hh_n}|\\
		&=\lim_{n\to \infty}|\ip{k_n}{\Delta_H^\frac{1}{2} h_n}|\\
		&=\lim_{n\to \infty}|\ip{\Delta_H^\frac{1}{4}k_n}{\Delta_H^\frac{1}{4} h_n}|\\
		&\leq \lim_{n\to\infty}\|\Delta_H^\frac{1}{4}k_n\| \|\Delta_H^\frac{1}{4}h_n\|\\
		&= 0,
	\end{align*}
 contradicting \eqref{eq:awayfromzero}. Hence, $\Delta_H^{1/4}E_K$ is not compact.
\end{proof}

We mention in passing that this result has applications in the analysis of twisted Araki-Woods algebras \cite{CorreaDaSilvaLechner:2023}. Also note that Prop.~\ref{prop:notcompactEK} implies that modular compactness for $K\subset H$ is equivalent to modular compactness of $H'\subset K'$.

\section{Inclusions and Gelfand Triples}\label{section:rigged}

Given a standard subspace $H\in \Std(\Hil)$, it is often not easy to construct $K\in \Std(H)$, as discussed in Section~\ref{section:inclusions}, \emph{e.g.} Lemma~\ref{lemma:extendK}. We proceed describing a method to not only construct examples of standard subspaces $K\subset H$, but also characterize when a (closed real) subspace $K\subset H$ is standard.

We work in the setting of Gelfand triples defined in terms of the modular operator of a standard subspace $H$, which is fixed throughout this section unless said otherwise. For that, we denote $\Hil_{H,+}:=\Cl[H]=\Dom(\Delta_H^{\nicefrac{1}{2}})\subset \Hil$ endowed with the graph inner product
$$\ip{\varphi}{\psi}_{H,+}:=\ip{(1+\Delta_H)^{\nicefrac{1}{2}}\varphi}{(1+\Delta_H)^{\nicefrac{1}{2}}\psi}=\ip{\varphi}{\psi}+\langle \Delta_H^{\nicefrac{1}{2}} \varphi ,\Delta_H^{\nicefrac{1}{2}}\psi\rangle.$$
It is well known that $\Hil_{H,+}$ is a Hilbert space since $\Delta_H^{\nicefrac{1}{2}}$ is closed. Similarly, we denote by $\Hil_{H,-}$ the completion of $\Ran((1+\Delta_H)^\frac{1}{2})\subset \Hil$ with respect to the norm induced by the inner product $$\ip{\varphi}{\psi}_{H,-}:=\ip{(1+\Delta_H)^{-\nicefrac{1}{2}}\varphi}{(1+\Delta_H)^{-\nicefrac{1}{2}}\psi}.$$

\begin{theorem}
	Let $H\subset \Hil$ be a standard subspace. Then, the inclusion
	\begin{equation}
		\Hil_{H,+}\subset \Hil \subset \Hil_{H,-}
	\end{equation}
	holds and this triple is a Gelfand triple (or a rigged Hilbert space), namely, both the inclusions $\Hil_{H,+}\hookrightarrow \Hil$ and $\Hil \hookrightarrow \Hil_{H,-}$ are continuous and have dense range, and there exists a unitary operator $\Gamma_H:\Hil_{H,+}\to \Hil_{H,-}$ extending $(1+\Delta_H)$ and satisfying
	\begin{equation}\label{eq:Berezansky}
		\ip{\xi}{x}=\ip{\Gamma_H\xi}{x}_{H,-}=\ip{\xi}{\Gamma_H^{-1}x}_{H,+}, \qquad \forall \xi\in \Hil_{H,+}, \, \forall x\in \Hil.
	\end{equation}
\end{theorem}
\begin{proof}
	By the definitions of standard subspace and completion, respectively, we have that $\Hil_{H,+}=\Cl[H]$ and $\Ran((1+\Delta_H)^\frac{1}{2})\subset \Hil_{H,-}$ are dense. Moreover, by the positivity of the modular operator, the inequality  $\|\cdot\|_{H,-}\leq \|\cdot\|\leq \|\cdot\|_{H,+}$ holds and, hence, $\Ran((1+\Delta_H)^\frac{1}{2})\subset \Hil \subset \Hil_{H,-}$ and the inclusions are continuous.
	
	It follows from the definition of the spaces $\Hil_{H,+}$, $\Hil_{H,-}$ and their norms that the operator $(1+\Delta_H):\Dom(\Delta_H)\subset \Hil_{H,+}\to \Hil_{H,-}$ is an isometry with dense range. Thus, this operator has a unique unitary extension $\Gamma_H:\Hil_{H,+}\to \Hil_{H,-}$, sometimes called the Berezansky canonical isomorphism. Equation \eqref{eq:Berezansky} is an immediate consequence of the definition.
\end{proof}

Notice that the existence of the unitary operator $\Gamma_H$ satisfying \eqref{eq:Berezansky} make $(\Hil_{H,+},\Hil_{H,-})$ a dual pair with the pairing
$$\ip{\xi}{\varphi}_{+,-}:=\ip{\Gamma_H\xi}{\varphi}_{H,-}=\ip{\xi}{\Gamma_H^{-1}\varphi}_{H,+}, \quad \forall \xi\in \Hil_{H,+}, \ \varphi\in \Hil_{H,-}.$$
Another interesting point is that this dual pairing factorizes through $\Hil$, meaning that for $\xi \in \Hil_{H,+}$ we can consider the functional $\varphi_\xi: \Hil\subset \Hil_{H,-} \to \mathbb{C}$ defined by $\varphi_\xi(x)= \ip{\xi}{x}$. Since $$|\ip{\xi}{x}|=\left|\langle(1+\Delta_H)^{1/2}\xi,(1+\Delta_H)^{-1/2}x\rangle\right|\leq \|\xi\|_{H,+}\|x\|_{H,-},$$ $\varphi_\xi$ is a continuous functional on a dense subset of $\Hil_{H,-}$, therefore it can be extended to a functional $\hat{\varphi}_\xi$ on the whole space $\Hil_{H,-}$. This extension corresponds to a unique vector $\xi^*\in \Hil_{H,-}$ satisfying $\hat{\varphi}_\xi(x)=\ip{\xi^*}{x}_{H,-}$ for all $x\in \Hil_{H,-}$ and $\Gamma_H$ is exactly the mapping given by $\Gamma_H\xi=\xi^\ast$.

We emphasize here that $\Gamma_H^{-1}$ and $(1+\Delta_H)^{-1}$ coincide on $\Ran(1+\Delta_H)^{-1}$ and the latter clearly extends to a bounded operator in the whole Hilbert space $\Hil$. Therefore, it must hold that
\begin{equation}\label{eq:restBerezansky}
	\Gamma_H^{-1}\restr{\Hil}=(1+\Delta_H)^{-1},
\end{equation}
making the inverse of the Berezansky canonical isomorphism explicit in $\Hil$.

\begin{remark}
	For the sake of clarity, $^\perp$ always denotes the orthogonal complement of a set in $\Hil$ with respect to its given (complex) inner product; $^\rperp$, as defined in Section~\ref{section:standardsubspaces}, denotes the orthogonal complement with respect to the inner product $\Re\ip{\cdot}{\cdot}$ in $\Hil$ and, when considering the real orthogonal complement of a set in a real subspace, we will always explicit write it; $^{\perp_{H,\pm}}$ denotes the orthogonal complement of a set in $\Hil_{H\pm}$ with respect to the complex inner product $\ip{\cdot}{\cdot}_{H,\pm}$.
\end{remark}

The aim of the next lemma is to show not only that the orthogonal complement $K^{\perp_{H,+}}$ for $K\subset H$ can be explicitly calculated, but also that the inner product of $\Hil_{H,+}$ is convenient for considering inclusions $K\subset H$.

\begin{lemma}\label{lemma:orthogonalH1}
	Let $H$ be a standard subspace and $K\subset H$. Then,
	\begin{equation}\label{eq:H1orthogonal}
		K^{\perp_{H,+}}=\Cl[K^{\rperp}\cap H] \quad \text{and} \quad \overline{\Cl[K]}^{\|\cdot\|_{H,+}}=\Cl\left[\,\overline{K}\,\right].
	\end{equation}
\end{lemma}
\begin{proof}
	Let $h_1,h_2, h_3\in H$. A direct calculation shows that
	\begin{align*}
		\ip{h_1+ih_2}{h_3}_{H,+}&=\ip{h_1+ih_2}{h_3}+\langle \Delta_H^{\nicefrac{1}{2}}(h_1+ih_2),\Delta_H^{\nicefrac{1}{2}}h_3 \rangle\\
		&=\ip{h_1+ih_2}{h_3}+\ip{S_H h_3}{S_H(h_1+ih_2)}\\
		&=\ip{h_1+ih_2}{h_3}+\ip{h_3}{h_1-ih_2}\\
		&=2\Re \ip{h_1}{h_3}-2i\Re\ip{h_2}{h_3}.
	\end{align*}
	
	Taking $h_3=k\in K$, we conclude that $\ip{h_1+ih_2}{k}_{H,+}=0$ holds for all $k\in K$ if, and only if, $\Re\ip{h_1}{k}=0=\Re\ip{h_2}{k}$ for all $k\in K$ or, equivalently, $h_1,h_2\in K^{\rperp}\cap H$.
	
	The second equality follows by taking the orthogonal complement twice.
\end{proof}

Rigged Hilbert spaces proved to be a useful setup for considering the question of when a subspace $\CK\subset \Hil_{H,+}\subset \Hil$ is dense in $\Hil$. Showing this characterization is the purpose of the next result, which is a version of \cite[Theorem 6.1.4]{KoshmanenkoDudkin:2016} adapted to the context of standard subspaces.

\begin{proposition}\label{prop:orthogonalK}
	Let $H$ be standard subspace and $K\subset H$ a subspace. Then,
	\begin{gather}
		K^\perp=\Hil\cap \Gamma_H(K^{\perp_{H,+}}) \quad \text{and}\\
		\Gamma_H^{-1}(K^\perp)=K^{\perp_{H,+}}\cap \Dom(\Delta_H)=\Cl[K^\rperp\cap H]\cap \Cl[H']
	\end{gather}
	
\end{proposition}
\begin{proof}
	Thanks to \eqref{eq:Berezansky}, it holds that, for every $k\in K\subset \Hil_{H,+}$ and $\psi \in \Hil$,
	\begin{align*}
		\ip{k}{\psi}=\ip{k}{(\Gamma_H)^{-1}\psi}_{H,+}
	\end{align*}
	Therefore, $\psi\in K^\perp$ if, and only if, $(\Gamma_H)^{-1}\psi\in K^{\perp_{H,+}}$, which is equivalent to  $\psi \in\Hil\cap\Gamma_H(K^{\perp_{H,+}})$. Using \eqref{eq:restBerezansky}, follows $\Gamma_H^{-1}(K^\perp)=K^{\perp_{H,+}}\cap \Dom(\Delta_H)$.
	
	The last identification follows from noticing that $K^{\perp_{H,+}}\subset \Hil_{H,+}=\Dom(\Delta_H^{\nicefrac{1}{2}})$ and, thanks to Lemma \ref{lemma:orthogonalH1}, $S_H(K^{\perp_{H,+}})=S_H(\Cl[K^\rperp\cap H])=\Cl[K^\rperp\cap H]$. Thus, $\kappa\in\Cl[K^\rperp\cap H]\cap \Dom(\Delta_H)$ is equivalent to $S_H\kappa\in \Cl[K^\rperp\cap H]\cap \Dom(S_H^\ast)=\Cl[K^\rperp\cap H]\cap\Cl[H']$.
\end{proof}

An immediate consequence of Proposition \ref{prop:orthogonalK}, and that we emphasize as a corollary, characterizes when a closed subspace $K\subset H$ is a standard subspace.
\begin{corollary}\label{cor:characterizationcyclicsubspace}
	Let $H$ be standard subspace and $K\subset H$ a subspace. Then, $\overline{K}$ is a standard subspace if and only if
	\begin{equation}
		\Hil\cap \Gamma_H(K^{\perp_{H,+}})=\{0\}=\Cl[\, K^\rperp\cap H\,]\cap \Cl[H'].
	\end{equation}
\end{corollary}

As brought up in Section~\ref{subsec:Realorthogonalpolarizer}, the inclusion $K\subset H$ can be seen as an inclusion of a closed subspace $K$ in the real Hilbert space $H$ endowed with the inner product $\Re\ip{\cdot}{\cdot}$. Therefore, $H=K\oplus (K^\rperp\cap H)$. That tells us that all inclusions of standard subspaces are constructed by either enlarging a given standard subspace $K$ by defining $H=K\oplus_\Rl F$ or reducing the given standard subspace $H$ by setting $K=(H'\oplus_\Rl \hat{F})'=H\ominus_\Rl F$, where $F,\hat{F}\subset \Hil$ are closed subspaces. This idea has already appeared in Lemma~\ref{lemma:extendK} and Proposition~\ref{prop:singular-extension}, but here we employ the rigged space methods to obtain more general results.

\begin{proposition}\label{prop:extensions} Let $F,\hat{F}\subset \Hil$ be closed real subspaces.
	\begin{enumerate}
		\item Let $K\subset \Hil$ be a standard subspace. Then, $H:=\overline{K+F}$ is separating if and only if $\Cl[H'^{\perp_\Rl}\cap K']\cap \Cl[K]=\{0\}$. Moreover, $H'^{\perp_\Rl}\cap K'=\overline{E_{K'}(iF)}$ and, therefore, $K'=H'\oplus \overline{E_{K'}(iF)}$.
		\item Let $H\subset \Hil$ be a standard subspace. Then, $K:=(H'+\hat{F})'\subset H$ is cyclic if and only if $\Cl[K^{\perp_\Rl}\cap H]\cap \Cl[H']=\{0\}$. Moreover, $K^{\perp_\Rl}\cap H=\overline{E_H(i\hat{F})}$ and, therefore, $H=K\oplus \overline{E_H(i\hat{F})}$.
	\end{enumerate}
\end{proposition}
\begin{proof}
	$a)$ $H=\overline{K+F}=K\oplus \overline{E_{K^{\perp_\Rl}}F}$ is separating if and only if $H'=K'\cap F$ is cyclic. Thanks to the criteria in Corollary \ref{cor:characterizationcyclicsubspace}, $H'=K'\cap F\subset K'$ is cyclic if and only if $\Cl[H'^{\perp_\Rl}\cap K']\cap \Cl[K]=\Cl[(iH)\cap K']\cap \Cl[K]=\{0\}$.
	
	A vector $k'\in K'$ lies in $iH$ if, and only if, $k'=ih=\lim_{n\to\infty}ik_n+if_n$ for $h\in H$, $k_n\in K$, and $f_n\in F$, for all $n\in \Nl$. Hence, $k'=E_{K'} k'=\lim_{n\to\infty}E_{K'}(if_n)$ and $(iH)\cap K'\subset \overline{E_{K'}(iF)}$. Conversely, if $k'\in E_{K'}(iF)$, there exists $f\in F$ such that $k'=E_{K'}(if)$ and, of course, $if$ has an orthogonal decomposition $if=k'+ik$ with $k\in K$. Hence, $k'=i(-k+f)\in i(K+F)=iH$. Thus, $E_{K'}(iF)\subset (iH)\cap K'$ holds. The inclusion for the closure follows from $(iH)\cap K'$ already been closed and together with the previous inclusion yields equality.
	
	$b)$ follows from $a)$ by taking symplectic complements and replacing $F\leftrightarrow\hat{F}$.
	
\end{proof}

The results above admit an easier version when we start with $F\subset K^{\perp_\Rl}$ and $\hat{F}\subset H'^{\perp_\Rl}$. In this case, $F=K^{\perp_\Rl}\cap H$ and $\hat{F}=H'^{\perp_\Rl}\cap K'$, and the previous statements can be replaced by their simplified forms
\begin{enumerate}[a')]
	\item Let $K\subset \Hil$ be a standard subspace. Then, $H:=\overline{K+F}$ is separating if and only if $\Cl[F]\cap \Cl[K]=\{0\}$
	\item Let $H\subset \Hil$ be a standard subspace. Then, $K:=(H'+\hat{F})'\subset H$ is cyclic if and only if $\Cl[\hat{F}]\cap \Cl[H']=\{0\}$
\end{enumerate}

As already discussed, an important question in applications is not only when $K\subset H$ is a standard subspace, but also when the relative symplectic complement $K'\cap H$ of the inclusion is a standard subspace, \emph{i.e.} replacing $K$ by $K'\cap H$. For that, we are going to use the results of this section together with Section \ref{section:inclusions} to obtain conditions for the relative symplectic complement be standard.

\begin{lemma}
	Let $H$ be a standard subspace and $K\subset H$. Then, $$(K'\cap H)^{\perp_{H,+}}=\Cl\left[\overline{(\Delta_H-1)(\Delta_H+1)^{-1}K}\right].$$
\end{lemma}
\begin{proof}
	Lemma \ref{lemma:orthogonalH1} together with Proposition \ref{lemma:relativesymp} say that \begin{align*}
		(K'\cap H)^{\perp_{H,+}}&=\Cl[(K'\cap H)^{\rperp}\cap H]\\
		&=\Cl[(\Ran(D_HE_K)^\rperp\cap H)^{\rperp}\cap H]\\
		&=\Cl\left[\overline{\Ran(D_HE_K)}\right]\\
		&=\Cl\left[\overline{(\Delta_H-1)(\Delta_H+1)^{-1}K}\right]
	\end{align*}
\end{proof}

\begin{proposition}\label{prop:cyclicrelativesymp}
	Let $H$ be a factorial standard subspace and $K\subset H$ a subspace. Then,  $K'\cap H$ is standard if, and only if, one (hence both) of the following equalities hold $$\Cl[\overline{D_H K}]\cap \Cl[H']=\{0\}=\Hil\cap\Cl\left[\overline{\left(\Gamma_H-2\right)K}^{\|\cdot\|_{H,-}}\right]$$
\end{proposition}
\begin{proof}
	Corollary \ref{cor:characterizationcyclicsubspace} tells us that $K'\cap H$ is standard if and only if \begin{align*}
		\{0\}&=\Hil \cap \Gamma_H(K'\cap H)^{\perp_{H,+}}\\
		&=\Hil \cap \Cl\left[\Gamma_H\overline{(\Delta_H-1)(\Delta_H+1)^{-1}K}\right]\\
		&=\Hil \cap \Cl\left[\overline{\Gamma_H\left(\Delta_H+1\right)^{-1}\left(\Delta_H-1\right)K}^{\|\cdot\|_{H,-}}\right]\\
		&=\Hil \cap \Cl\left[\overline{\left(\Gamma_H-2\right)K}^{\|\cdot\|_{H,-}}\right]
	\end{align*}
	where we used that $\Gamma_H$ is unitary and $\Gamma_H\restr{\Dom(\Delta_H)}=1+\Delta_H$. From the second equality and \eqref{eq:restBerezansky} it follows that $\{0\}=\Cl[\overline{D_H K}]\cap \Dom(\Delta_H)=\Cl[\overline{D_H K}]\cap \Cl[H']$.
\end{proof}

Given $H\in \Std(\Hil)$, we have produced in Proposition~\ref{prop:singular-extension} examples of singular inclusions whenever $\Delta_H$ and $P_H$ are both unbounded. We proceed now, in the spirit of  Proposition \ref{prop:extensions}, by explicitly calculating the relative symplectic complement for standard subspaces which are enlarged $H:=\overline{K+F}$ or reduced $K:=(H'+\hat{F})'$ by closed subspaces $F$ or $\hat{F}$, respectively. Here, the central role of the cutting projection will become evident.

\begin{proposition}\label{prop:relsimpextensions} Let $F,\hat{F}\subset \Hil$ be closed real subspaces.
	\begin{enumerate}
		\item Let $K\subset \Hil$ be a standard subspace and define $H:=\overline{K+F}$. Then, $\Ran(P_K E_F)\subset K'\cap H$. If in addition $F\subset K^\rperp$, equality holds.
		\item Let $H\subset \Hil$ be a standard subspace and define $K:=(H'+\hat{F})'\subset H$. Then,  $\Ran(P_H E_{\hat{F}})\subset K'\cap H$. If in addition $\hat{F}\subset H'^\rperp$, equality holds.
	\end{enumerate}
\end{proposition}
\begin{proof}
	Items $a)$ and $b)$ are equivalent by taking symplectic complements, hence it is enough to prove $a)$. For the first inclusion, for any vector $h\in \Ran(P_K E_F)$ there exists $f\in F \cap \dom(P_H)$ such that $h=P_H(f)$. Such a vector $f$ must have a unique decomposition $f=h+h'$ because $f\in \Dom(P_H)=H+H'$. Hence, $h=-h'+f\in K'\cap H$ and $\Ran(P_K E_F)\subset K'\cap H$ follows. For the second inclusion, we notice that if $F\subset K^\rperp$, $H=K+F$ is already closed. Thus, any vector $k'\in K'\cap H$ is of the form $k'=k+f$, where $k'\in K'$, $k\in K$ and $f\in F$. This implies that $f=k'-k\in \Dom(P_K)$ and $P_K(f)=k'$. Thus $K'\cap H\subset \Ran(P_K E_F)$.
\end{proof}

We remark that even the inclusion $\overline{\Ran(P_H E_{\hat{F}})}\subset K'\cap H$ always holds, but more interestingly, when $\hat{F}\subset H'^\rperp$, we have that $\Ran(P_H E_{\hat{F}})= K'\cap H$, \emph{i.e.} $P_H E_{\hat{F}}$ has closed range.

Notice that, by combining Proposition \ref{prop:extensions} and Corollary \ref{cor:characterizationcyclicsubspace}, we obtain that, given a standard subspace $H\subset \Hil$ and a closed subspace $\hat{F}\subset H'^\rperp$, the condition for $K:=(H'+\hat{F})'=H\ominus\overline{E_Hi\hat{F}}\subset H$ to be cyclic reduces to requiring that $\hat{F}$ satisfy $\Cl[\overline{E_Hi\hat{F}}]\cap \Cl[H']=\{0\}$. Now, if we also require $K'\cap H$ to be cyclic, Proposition~\ref{prop:relsimpextensions} and Corollary~\ref{cor:characterizationcyclicsubspace} lead to the conclusion that $\hat{F}$ must also satisfy the condition $$\{0\}=\Cl[\Ran(P_H E_{\hat{F}})^\rperp\cap H]\cap \Cl[H']=\Cl[\ker(E_{\hat{F}}P_{iH})]\cap H]\cap \Cl[H'].$$

\section{Inclusions and Extensions of the Modular Operator}\label{section:extensions}

In this section, we describe standard subspaces $K\in\Std(H)$ of a given standard subspace $H\in\Std(\Hil)$ by operator-theoretic methods. The main idea is to consider the linear densely defined operator
\begin{align}
	T_{K,H}:=S_K^*S_H.
\end{align}
Since $S_K^*=T_{K,H}S_H$, all properties of the inclusion $K\subset H$ are in principle encoded in $T_{K,H}$ (for fixed $H$).

We consider a set of extensions $(T,\dom T)$ of $\Delta_H$ with further properties, namely
\begin{align*}
	\Ext(\Delta_H)
	:=
	\{(T,\dom T)\col T\supset\Delta_H,\; TS_H\, \text{is closable},\,\dom\Delta_H^{1/2}\supset\dom T \}.
\end{align*}
Of course, a proper extension $T$ of the selfadjoint operator $\Delta_H$ cannot be symmetric, and we also do not ask $T$ to be closable. A special role will be played be the more particular extensions in
\begin{align}
	\Ext_0(\Delta_H)
	:=
	\{T\in\Ext(\Delta_H)\col \overline{TS_H}\,\text{ is an involution}\}.
\end{align}

\begin{proposition}
	Let $H\in\Std(\Hil)$ and consider the maps
	\begin{alignat}{2}
		\ex_H:\Std(H)&\to\Ext(\Delta_H),\qquad  \ex_H(K)&&:=T_{K,H}:=S_K^*S_H,
		\\
		\std_H:\Ext(\Delta_H)&\to\Std(H),\qquad \std_H(T)&&:=\Fix(\overline{TS_H})'.
	\end{alignat}
	These maps are well-defined and satisfy
	\begin{align}\label{eq:ex-std-halfinverse}
		\std_H\circ\ex_H=\id_{\Std(H)}.
	\end{align}
	The range of $\ex_H$ is $\Ext_0(\Delta_H)$, so its corestriction $\ex_H^0:\Std(H)\to\Ext_0(\Delta_H)$ is bijective with inverse $(\ex_H^0)^{-1}={\std_H}|_{\Ext_0(\Delta_H)}$.
\end{proposition}
\begin{proof}
	Let $K\in\Std(H)$. Then it is clear that $\ex_H(K)=S_K^*S_H$ is a linear operator. Since $K\subset H$, we have $K'\supset H'$ and hence $S_K^*\supset S_H^*$. This implies $S_K^*S_H\supset S_H^*S_H=\Delta_H$. Furthermore $TS_H=S_K^*$ is a closed operator extending $S_H^*$. Hence the map $\ex_H$ is well-defined.

	Regarding the map $\std_H$, let $T\in\Ext(\Delta_H)$ and consider the space $K:=\Fix(\overline{TS_H})'$. Clearly, $K$ is a closed real linear subspace of~$\Hil$. To show $K\in\Std(H)$, we first observe that $\Fix(\overline{TS_H})$ is separating because $\overline{TS_H}$ is antilinear. Thus its symplectic complement $K$ is cyclic. Furthermore, since $\overline{TS_H}\supset S_H^*$, we have $\overline{TS_H}h'=h'$ for $h'\in H'$, i.e. $H'\subset K'$ and thus $K\subset H$. So we see that $K$ is included in $H$ and therefore separating, hence standard, and the map $\std_H$ is also well-defined.

	For $K\in\Std(H)$, we have
	\begin{align}
		\std_H(\ex_H(K))
		=
		\Fix(\overline{S_K^*S_HS_H})'
		=
		\Fix(S_K^*)'
		=
		\Fix(S_K)
		=
		K,
	\end{align}
	proving \eqref{eq:ex-std-halfinverse}. In particular, $\ex_H$ is injective and $\std_H$ is surjective.

	Moving on to the proof of the last statement, for $K\in\Std(H)$ we clearly see that $\ex_H(K)S_H=S_K^*S_H^2=S_K^*$ is a closed involution and $\dom(S_K^*S_H)\subset\dom(\Delta_H^{\nicefrac{1}{2}})$, i.e. the range of $\ex_H$ is contained in $\Ext_0(\Delta_H)$. For $T\in\Ext_0(\Delta_H)$, the fixed point space $\Fix(\overline{TS_H})$ is closed, and we have
	\begin{align}
		\tilde T:=
		\ex_H(\std_H(T))
		=
		S_{\Fix(\overline{TS_H})'}^*S_H
		=
		S_{\Fix(\overline{TS_H})}S_H,
	\end{align}
	with domain $\dom\tilde T=\{v\in\dom S_H\col S_Hv\in\dom S_{\Fix(\overline{TS_H})}\}=\dom T$ and action $\tilde Tv=S_{\Fix(\overline{TS_H})}S_Hv=Tv$, i.e. $\tilde T=T$. This shows $\ex_H\circ{\std_H}|_{\Ext_0(\Delta_H)}=\id_{\Ext_0(\Delta_H)}$ and hence the claim.
\end{proof}

Via this proposition, extensions of $\Delta_H$ are seen to produce standard subspaces inside $H$ and vice versa. In case $\Delta_H$ is bounded, $\Ext(\Delta_H)=\{\Delta_H\}$, but in case~$\Delta_H$ is unbounded, it has many extensions.

We now study how relevant properties of the inclusion $K\subset H$ are reflected in the operator $T_{K,H}\supset\Delta_H$.

\begin{proposition}\label{prop:T}
	Let $H\in\Std(\Hil)$ and $K,K_1,K_2\in\Std(H)$.
	\begin{enumerate}
		\item $T_{K,H}$ is invertible with ${T_{K,H}}^{-1}=T_{H',K'}$.

		\item $K_1\subset K_2 \Longleftrightarrow T_{K_2,H} \subset T_{K_1,H}$ and $T_{H,H}=\Delta_H$.

		\item $T_{K,H}$ is bounded if and only if $\Delta_H$ is bounded (in this case $K=H$).
		\item Let $\la\in\Cl$, $|\la|=1$. The eigenspace for eigenvalue $\la$ of $T_{K,H}$ is
		\begin{align}\label{Tfix}
			\ker(T_{K,H}-\la)=\Cl[(\la^{-\nicefrac{1}{2}}K)'\cap H].
		\end{align}
		In particular,
		\begin{align}
		 	\Fix(T_{K,H})&=\Cl[K'\cap H],
		 	\qquad
 			\Fix(-T_{K,H})=\Cl[K^{\perp_\Rl}\cap H].
		\end{align}
		\item $T_{K,H}$ is closable if and only if $K=H$.
	\end{enumerate}
\end{proposition}
\begin{proof}
	a) is clear by definition of $T_{K,H}$.

	b) An inclusion $K_1\subset K_2$ of standard subspaces is equivalent to $S_{K_2}^*\subset S_{K_1}^*$ which is equivalent to $S_{K_2}^*S_H=T_{K_2,H}\subset T_{K_1,H}=S_{K_1}^*S_H$. The trivial inclusion $H=H$ gives the trivial extension $T_{H,H}=S_H^*S_H=\Delta_H$.

	c) If $\Delta_H$ is bounded, we have $K=H$ (Lemma~\ref{lemma:extendK}). In this case $\dom\Delta_H=\Hil$, so also $T_{K,H}\supset\Delta_H$ is bounded. If $T_{K,H}$ is bounded, then so is $\Delta_H\subset T_{K,H}$.

	d) The equality of domains \eqref{Tfix} is a known result for $\la=1$ \cite[Prop.~4.1]{BrunettiGuidoLongo:2002}, but we repeat the proof here for self-containedness. The inclusion $\Cl[K'\cap H]\subset\Fix(S_K^*S_H)$ is obvious. Conversely, $v\in\Fix(S_K^*S_H)$ lies in $\dom S_H$ and $\dom S_K^*$, i.e. $v=k_1'+ik_2'$ with $k_j'\in K'$. The fixed point property yields $S_Hv=S_K^*v=k_1'-ik_2'$, therefore $k_1'=\frac12(v+S_Hv)\in K'\cap H$ and $k_2'=\frac{1}{2i}(v-S_Hv)\in K'\cap H$.

	For general $\la$ with $|\la|=1$, we have
	\begin{align*}
		\Cl[(\la^{-\nicefrac{1}{2}}K)'\cap H]
		&=
		\Fix(T_{\la^{-\nicefrac{1}{2}}K,H})
		=
		\Fix((\la^{-\nicefrac{1}{2}}S_K\overline{\la^{-\nicefrac{1}{2}}})^*S_H)
		\\
		&=
		\Fix[\la^{-1}T_{K,H}]
		=
		\ker(T_{K,H}-\la)
		.
	\end{align*}
	The second statement follows directly from setting $\la=\pm1$ and $K^{\perp_\Rl}=iK'$.

	e) We first observe that $T_{K,H}$ is not only an extension of $\Delta_H$, but also of $\Delta_K$ because $T_{K,H}=S_K^*S_H\supset S_K^*S_K=\Delta_K$. $T_{K,H}$ is closable if and only if $T_{K,H}^*$ is densely defined. But $T_{K,H}^*\subset\Delta_H$, $T_{K,H}^*\subset\Delta_K$, and if $\D:=\dom(T_{K,H}^*)$ is dense, then the bounded operator $\Delta_K^{-1}\Delta_H$ extends the identity on $\D$. If $\D$ is dense, we have $\Delta_K^{-1}\Delta_H=1$. This implies $\Delta_K^{it}=\Delta_H^{it}$, $t\in\Rl$, and $K=H$ because these spaces form an inclusion. The converse is clear because $T_{H,H}=\Delta_H$.
\end{proof}

The above proposition shows that the extensions $T_{K,H}\supset\Delta_H$ must have quite delicate properties, in particular they are not closable and have dense eigenspace for eigenvalue $+1$ in case $K'\cap H$ is cyclic. No useful spectral information on $T_{K,H}$ is available beyond the point spectrum results from part~d).

We will therefore move on to a situation that is governed by operators with better (although still delicate) properties, essentially by taking a kind of square root of $T_{K,H}$. To this end, we specialize to unitary inclusions $K=UH\in\Std(H)$, $U\in\CE(H)$, and write
\begin{align}
	T_U:=T_{UH,H}=US_H^*U^*S_H
\end{align}
when $H$ is clear from the context. Recall that an endomorphism $U\in\CE(H)$ is called symmetric if $U=U^\#:=J_HU^*J_H$.

\begin{lemma}\label{lemma:BU}
	Let $H\in\Std(\Hil)$, $U\in\CE(H)$, and consider
	\begin{align}
		B_U :=  U^*\Delta_H^{\nicefrac{1}{2}}.
	\end{align}
	\begin{enumerate}
		\item\label{B:basics} $B_U$ is a densely defined closed operator with $\dom B_U=\dom\Delta_H^{\nicefrac{1}{2}}$ and $B_U^*=\Delta_H^{\nicefrac{1}{2}}U=U^*B_{U^*}U$.
		\item\label{B:symmetry} $B_U$ is symmetric as an operator if and only if $U$ is symmetric as an endomorphism.
		\item\label{B:selfadjointness} $B_U$ is selfadjoint if and only if $U$ is a selfadjoint (and unitary) automorphism. In this case $UH=H$.
		\item\label{B:positivity} If $U=V^2$ with $V\in\CE(H)$ symmetric, then $B_U$ has a selfadjoint positive extension, namely $B_U\subset V^*\Delta_H^{\nicefrac{1}{2}}V$.
		\item\label{B:andT} One has
		\begin{align}
			T_U = U(B_U^*)^2U^*
		\end{align}
		if and only if $U$ is symmetric. For any $|\la|=1$, we have
		\begin{align}
			U\ker((B_U^*)^2-\la)=\ker(T_{U}-\la).
		\end{align}
		In particular,
		\begin{align}
			U\Fix((B_U^*)^2)
			&=\Fix(T_U)=\Cl[UH'\cap H],
			\\
			U\Fix(-(B_U^*)^2)
			&=\Fix(-T_U)=\Cl[UH^\rperp\cap H],
			.
		\end{align}

		\item\label{B:andCyclicity} $UH'\cap H$ is cyclic if and only if $\Ran((B_U^*B_U^*)^*-1)=\{0\}$.
	\end{enumerate}
\end{lemma}
\begin{proof}
	\ref{B:basics} is clear. \ref{B:symmetry} $B_U$ is symmetric if and only if $B_U=U^*\Delta_H^{\nicefrac{1}{2}}\subset\Delta_H^{\nicefrac{1}{2}}U=B_U^*$, which is equivalent to $U^\#S_H\subset S_HU$. This shows that $B_U$ is symmetric if $U^\#=U$. Conversely, if $B_U$ is symmetric, both $S_HU$ and $S_HU^\#$ are extensions of $U^\#S_H$, and hence $U$ as well as $U^\#$ extend $S_HU^\#S_H$. As the latter operator is densely defined, and $U$, $U^\#$ are bounded, this implies $U=U^\#$.

	\ref{B:selfadjointness} Note that $B_U=U^*\Delta_H^{\nicefrac{1}{2}}$ is the polar decomposition of $B_U$, so $U$ is the partial isometry (unitary) of the polar decomposition of $B_U^*$. Thus selfadjointness of~$B_U$ implies $U=U^*$. Now $UH\subset H$ and unitarity of $U$ yields $UH=H$. Conversely, if $U=U^*$, then $UH=H$, and it follows that $U$ commutes with $\Delta_H^{\nicefrac{1}{2}}$ and hence $B_U^*=\Delta_H^{\nicefrac{1}{2}}U=U\Delta_H^{\nicefrac{1}{2}}=U^*\Delta_H^{\nicefrac{1}{2}}=B_U$.

	\ref{B:positivity} Under the stated conditions, we have
	\begin{align*}
		B_U
		=
		V^*V^*\Delta_H^{\nicefrac{1}{2}}
		=
		V^*J_HVJ_H\Delta_H^{\nicefrac{1}{2}}
		\subset
		V^*\Delta_H^{\nicefrac{1}{2}}V.
	\end{align*}

	\ref{B:andT} We have
	\begin{align*}
		T_U=UJ_H\Delta_H^{-\nicefrac{1}{2}}U^*J_H\Delta_H^{\nicefrac{1}{2}}&=U\Delta_H^{\nicefrac{1}{2}}J_HU^*J_H\Delta_H^{\nicefrac{1}{2}},
		\\
		U(B_U^*)^2U^*
		&=U\Delta_H^{\nicefrac{1}{2}}U\Delta_H^{\nicefrac{1}{2}},
	\end{align*}
	which yields the claim by comparison. The statement about the fixed point spaces is a direct consequence of Prop.~\ref{prop:T}~d).

	\ref{B:andCyclicity} We have $\Cl[K'\cap H]=U\ker((B_U^*)^2-1)$ and hence cyclicity is equivalent to
	\begin{align*}
		\{0\}
		&=
		\ker((B_U^*)^2-1)^\perp
		=
		\overline{\Ran((B_U^*B_U^*)^*-1)}.
	\end{align*}
\end{proof}

The operator $B_U$ can be seen as a kind of ``square root'' of $T_U$. In fact, the related operator $B_{U^*}=U\Delta_H^{\nicefrac{1}{2}}=UB_U^*U^*$ is even closer to being a square root because $B_{U^*}^2=T_U$. However, since $B_{U^*}^*B_{U^*}=\Delta_H$ and $B_{U^*}B_{U^*}^*=U\Delta_H U^*$, this operator is not normal for non-trivial inclusions. Hence we rather base our analysis on~$B_U$.

We also remark that the ``square situation'' in part~\ref{B:positivity} automatically holds if~$U$ is part of a one-parameter group of symmetric endomorphisms.

In a special situation that we will discuss further below, the operator $U^*\Delta_H^{\nicefrac{1}{2}}$ is the so-called bound state operator of an integrable quantum field theory, studied by Tanimoto \cite{Tanimoto:2015}. We therefore denote this operator $B_U$; parts of our analysis generalize results of \cite{Tanimoto:2015}.

\bigskip

We saw above that for a nontrivial inclusion $UH\subsetneq H$ with symmetric $U$, the symmetric operator $B_U$ is not selfadjoint. Since $B_U$ is closed, this is also equivalent to $B_U$ not being essentially selfadjoint, i.e. to at least one of its deficiency indices
\begin{align}
	n^\pm_U:=\dim\ker(B_U^*\mp i)
\end{align}
being larger than zero. This observation can be used to estimate the relative complement of the inclusion and obtain a non-singularity result.

\begin{corollary}\label{cor:DefIndBounds}
	Let $H\in\Std(\Hil)$ and $U\in\CE(H)$ a proper endomorphism, $UH\subsetneq H$.
	\begin{enumerate}
		\item\label{relorthcompl} If $U=U^\#$ is symmetric,
		\begin{align}
			\dim((UH)^{\perp_\Rl}\cap H) \geq n^+_U+n^-_U>0.
		\end{align}
		\item\label{relsymcompl} If $U=V^2$ with a symmetric endomorphism $V\in\CE(H)$, then $n^+_U=n^-_U$ and
		\begin{align}
			\dim((UH)'\cap H) \geq n^+_U>0.
		\end{align}
		In particular, $UH\subset H$ is not singular in this case.
	\end{enumerate}
\end{corollary}
\begin{proof}
	\ref{relorthcompl} We have $\Cl[UH^{\perp_\Rl}\cap H]=U\Fix(-(B_U^*)^2)$ by Lemma~\ref{lemma:BU}~e). Since $\ker(B_U^*\mp i)\subset\Fix(-(B_U^*)^2)$, the inequality $\dim((UH)^{\perp_\Rl}\cap H) \geq n^+_U+n^-_U$ follows. As $UH\neq H$, the operator $B_U$ is not essentially selfadjoint, so $n_U^++n_U^->0$.

	\ref{relsymcompl} By Lemma~\ref{lemma:BU}~\ref{B:positivity} we know that $B_U$ has positive spectrum under the given circumstances. Hence the function $f(\la):=\dim\ker(B_U^*-\la)$ is constant on $\Cl\backslash[0,\infty)$ \cite[Corollary to Thm.~X.1]{ReedSimon:1975}. In particular, $f(i)=n^+_U=f(-i)=n^-_U=f(-1)$, i.e. the eigenspace for eigenvalue $-1$ of $B_U^*$ has dimension $n^+_U$. As this space is contained in $\Fix((B_U^*)^2)=\Cl[UH'\cap H]$, the claimed bound follows.
\end{proof}

We remark that in case $U\in\CE(H)$ is symmetric but not a square, it is possible to have $UH\subsetneq H$ and $UH'\cap H=\{0\}$. An example of this is given in Theorem~\ref{thm:standardpairsrelativecomplements}~\ref{item:finiteblaschke}.

We next add another corollary that gives sharper results for decomposable endomorphisms.

\begin{corollary}\label{cor:dimestimates}
	Let $U_1,\ldots,U_n\in\CE(H)$ be symmetric proper endomorphisms, and $U=U_1\cdots U_n$. Then
	\begin{align}
		\dim((UH)^\rperp\cap H)\geq n.
	\end{align}
		If all $U_j$ squares of symmetric proper endomorphisms, $U_jH\neq H$, and $H$ is a factor, then
		\begin{align}
			\dim((UH)'\cap H)
			\geq n.
		\end{align}
		In particular, if $(U(x))_{x\in\Rl}$ is a one-parameter group of symmetric unitaries such that $U(x)H\subsetneq H$ for $x>0$, then $\dim(U(x)H'\cap H)=\infty$.
\end{corollary}
\begin{proof}
	The first two claims follow directly from combining Lemma~\ref{lemma:iteratedinclusions} and Corollary~\ref{cor:DefIndBounds}. The last claim is immediate by noting $U(x)=U(\frac{x}{n})^n$, $n\in\Nl$.
\end{proof}

Corollary~\ref{cor:dimestimates} will be used in the proof of Theorem~\ref{thm:standardpairsrelativecomplements}~\ref{item:singular}.

\section{Inclusions Arising from Endomorphisms of Standard Pairs}\label{section:standardpairsandinnerfunctions}

In this section we introduce a family of specific examples of inclusions of standard subspaces and discuss its properties.

\subsection{The Irreducible Standard Pair} Let $(H,T)$ be a {\em non-degenerate standard pair}, consisting of a standard subspace $H\subset \Hil$ and a strongly continuous unitary one-parameter group $T(x)=e^{ixP}$ with positive generator $P>0$ with $\ker P=\{0\}$, such that
\begin{align}
	T(x)H\subset H,\qquad x\geq0.
\end{align}
Such standard subspaces play a prominent role in quantum field theory (see, for example \cite{Borchers:1992,GuidoLongoWiesbrock:1998} and half-sided modular inclusions (see \cite{Wiesbrock:1997,ArakiZsido:2005,Koot:2025} for some classical and recent articles).

These standard subspaces are factors, and the modular unitaries satisfy the exchange relation $\Delta_H^{it}T(x)\Delta_H^{-it}=T(e^{-2\pi t}x)$, $t,x\in\Rl$, i.e. they generate a representation of the affine group of $\Rl$. In case this representation is irreducible, we speak of an irreducible standard pair.

Up to unitary equivalence, there exists a unique irreducible non-degenerate standard pair. It can be realized as (see, for example, \cite{LechnerLongo:2014})
\begin{align*}
	\Hil&=L^2(\Rl_+,\tfrac{dp}{p}),\qquad (T(x)\psi)(p)=e^{ipx}\psi(p),\\
	(\Delta_H^{it}\psi)(p) &= \psi(e^{-2\pi t}p),\qquad\; (J_H\psi)(p)=\overline{\psi(p)}.
\end{align*}
In terms of the skew-symmetric extension
\begin{align}\label{eq:skew}
	\psi_{\tt s}(p)
	:=
	\begin{cases}
		\overline{\psi(-p)} & p<0\\
		\psi(p) & p>0
	\end{cases}
\end{align}
of $\psi\in\Hil$ to $\Rl$, the standard subspace $H$ consists precisely of all $\psi\in\Hil$ such that the support (in the sense of distributions) of the inverse Fourier transform $\CF^{-1}(\psi_{\tt s})$ is contained in $[0,\infty)$ \cite[Cor.~A.4]{LechnerLongo:2014}. Alternatively, $H$ is given as the closure of the Fourier transforms of real functions supported on the right half line, namely
\begin{align}\label{eq:HbyFourierTransform}
	H
	&=
	\{\CF(\partial f)|_{\Rl_+} \col f\in C_{c,\Rl}^\infty(\Rl),\;\supp f\subset[0,\infty)\}^{\|\cdot\|},
\end{align}
where $\partial f$ denotes the derivative of $f$  \cite[Eqn.~(A.4)]{LechnerLongo:2014}. Functions in $H$ are boundary values of functions analytic in the upper half plane and satisfying $\overline{\psi(-p)}=\psi(p)$, $p>0$.

A unitarily equivalent description is given by change of variables to $\te=\log p$, namely
\begin{align}
	\Hil&=L^2(\Rl,d\te),\qquad (T(x)\psi)(\te)=e^{ixe^\te}\psi(\te),\\
	(\Delta_H^{it}\psi)(\te) &= \psi(\te-2\pi t),\qquad\; (J_H\psi)(\te)=\overline{\psi(\te)}.
\end{align}
In this realization, $H$ can be described in terms of the Hardy space $\Hardy^2(\Strip_\pi)$ of the strip $\Strip_\pi=\{\zeta\in\Cl\col0<\Im\zeta<\pi\}$ as
\begin{align}
	\Hardy^2(\Strip_\pi)
	&=
	H+iH,\\
	H
	&=\{\psi\in\Hardy^2(\Strip_\pi) \col \overline{\psi(\te+i\pi)}=\psi(\te)\,\text{for a.e. } \te\in\Rl\}.
	\label{HStripPic}
\end{align}
We will refer to these two realizations of the irreducible standard pair as the upper half plane picture and the strip picture, respectively. Clearly, $\sigma(\Delta_H)=[0,\infty)$ has no spectral gap and no eigenvalues, so $H$ is a factor space of type~III${}_1$.

In the following proposition, we determine the projection and polarizer of~$H$.

\begin{proposition}\label{prop:realortproj}
	Let $H$ be the standard subspace \eqref{HStripPic} in $L^2(\Rl,d\te)$.
	\begin{enumerate}
		\item The real orthogonal projection $E_H:\Hil\to H$ onto $H$ is given by, $\psi\in L^2(\Rl,d\te)$
		\begin{align}
			(E_H\psi)(\te)
			&=
			\frac{i}{4\pi}\lim_{\eps\searrow0}
			\int_\Rl \frac{\psi(\te')}{\sinh\frac{\te-\te'+i\eps}{2}}d\te'
			+
			\frac{1}{4\pi}
			\int_\Rl \frac{\overline{\psi(\te')}}{\cosh\frac{\te-\te'}{2}}d\te'
			\,.
		\end{align}
		\item The polarizer of $H$ is the restriction to $H$ of the map
		\begin{align*}
			(D_H\psi)(\te)
			=
			i\psi(\te)-\frac{i}{2\pi}\lim_{\eps\searrow0}
			\int_\Rl \frac{\psi(\te')}{\sinh\frac{\te-\te'+i\eps}{2}}d\te'
			.
		\end{align*}
	\end{enumerate}
\end{proposition}
\begin{proof}
	a) This formula can be proven by computing $(1+\Delta_H)^{-1}$ directly via Fourier analysis and then applying $(1+S_H)$ to obtain $E_H=(1+S_H)(1+\Delta_H)^{-1}$. We prefer to give a conceptually more interesting argument based on graphs of modular operators.

	For $H\subset\Hil$ an arbitrary standard subspace in an arbitrary Hilbert space, consider the graph $G_H=\{\psi\oplus\Delta_H^{1/2}\psi\,:\,\psi\in\dom\Delta_H^{1/2}\}$, a closed subspace of $\Hil\oplus\Hil$. As is well known and can be easily verified, the orthogonal projection onto this subspace is given by
	\begin{align*}
		P_{G_H}(\varphi\oplus\psi)
		&=
		\xi\oplus\Delta_H^{1/2}\xi,\\
		\xi
		&=
		(1+\Delta_H)^{-1}\varphi+\Delta_H^{1/2}(1+\Delta_H)^{-1}\psi.
	\end{align*}
	Comparing with the real linear projection onto the standard subspace, namely $E_H=(1+\Delta_H)^{-1}+\Delta_H^{1/2}(1+\Delta_H)^{-1}J_H$, we see that $E_H$ can be written as the composition of the three maps
	\begin{equation}
	\begin{tikzcd}[row sep=small]
		\Hil \arrow{r}{1\oplus J_H} & \Hil\oplus\Hil \arrow{r}{P_{G_H}} &  G_H \arrow{r}{\Pi_1} & \dom\Delta_H^{1/2}
		\\%
		\psi \arrow[r, maps to] & \psi\oplus J_H\psi \arrow[r, maps to] & P_{G_H}(\psi\oplus J_H\psi) \arrow[r, maps to] & E_H\psi
	\end{tikzcd}
	\end{equation}
	where $\Pi_1:\Hil\oplus\Hil\to\Hil$ is the projection onto the first summand.

	In the concrete situation at hand, the graph can be identified with $G_H=\Hardy^2(\Strip_\pi)$ (see Section~\ref{section:rigged} or \cite{LechnerLiQueffelecPiazza:2018}), with the two summands in $\Hil\oplus\Hil=L^2(\Rl)\oplus L^2(\Rl)=L^2(\partial\Strip_\pi)$ corresponding to the boundary values at $\Rl$ and $\Rl+i\pi$. The orthogonal projection $L^2(\partial\Strip_\pi)\to\Hardy^2(\Strip_\pi)$ is given by integration of $\varphi\oplus\psi$ against the Szegö kernel $\mathscr K$ of $\Strip_\pi$ (see, for example, \cite{Krantz:1951}). This kernel is known to be \cite{LechnerLiQueffelecPiazza:2018}
	\begin{align*}
		\mathscr K(\zeta,\zeta')
		=
		\frac{i}{4\pi}\frac{1}{\sinh\frac{\zeta-\overline{\zeta'}}{2}},\qquad \zeta,\zeta'\in\Strip_\pi.
	\end{align*}
	Hence the projection onto the graph of $\Delta_H^{1/2}$ takes the form
	\begin{align*}
		P_{G_H}(\psi\oplus\eta)(\zeta)
		&=
		\int_\Rl \mathscr K(\zeta,\te')\psi(\te')d\te'+\int_\Rl \mathscr K(\zeta,\te'+i\pi)\eta(\te')d\te'
		\\
		&=
		\frac{i}{4\pi}\int_\Rl \frac{\psi(\te')}{\sinh\frac{\zeta-\te'}{2}}\,d\te'
		+
		\frac{i}{4\pi}\int_\Rl \frac{\eta(\te')}{i\cosh\frac{\zeta-\te'}{2}}\,d\te'
		\,.
	\end{align*}
	As explained above, $E_H\psi=\Pi_1P_{G_H}(\psi\oplus J_H\psi)$. Since $J_H$ acts by complex conjugation and $\Pi_1$ by taking the boundary values at the real line, the claimed formula follows.

	b) This is immediate from $D_H=i\frac{\Delta_H-1}{\Delta_H+1}|_H=i\left(1-2(\Delta_H+1)^{-1}\right)|_H$ and identifying $(1+\Delta_H)^{-1}$ as the linear term in a).
\end{proof}

The cutting projection of $H$ can be described by multiplying a real test function $f$ with $f(0)=0$ by the characteristic function of $\Rl_+$ and subsequent Fourier transform. We refrain from giving the details here.

In our investigations of standard subspaces included in $H$, we will also make use of the following symmetry.

\begin{lemma}\label{lemma:Y}
	Consider the antiunitary involution $Y:L^2(\Rl,d\te)\to L^2(\Rl,d\te)$, $(Y\psi)(\te)=\overline{\psi(-\te)}$ (strip picture) or $Y:L^2(\Rl_+,\frac{dp}{p})\to L^2(\Rl_+,\frac{dp}{p})$, $(Y\psi)(p)=\overline{\psi(p^{-1})}$ (half space picture). Then
	\begin{align}
		YH=H.
	\end{align}
\end{lemma}
\begin{proof}
	Given the characterization \eqref{HStripPic} of $H$ in the strip picture, it is clear that $Yh$, $h\in H$, extends as $(Yh)(\zeta)=\overline{h(-\overline{\zeta})}$, $\zeta\in\Strip_\pi$, and therefore maps $\Hardy^2(\Strip_\pi)$ onto $\Hardy^2(\Strip_\pi)$. For $h\in H$ and $\te\in\Rl$, we have $(Yh)(\te+i\pi)=\overline{h(-\te+i\pi)}=h(-\te)=\overline{Yh(\te)}$. Thus $YH=H$.

	The claim for the upper half plane picture follows by change of variables $p=e^\te$.
\end{proof}

\subsection{Endomorphisms of Standard Pairs and Associated Inclusions}

To describe the inclusions that we want to study, we consider endomorphisms $V\in\CE(H)$ that commute with the translation unitaries $T(x)$, $x\in\Rl$. This restricted endomorphism semigroup is \cite{LongoWitten:2010}
\begin{align*}
	\CE(H,T)
	&:=
	\{V\in\CE(H)\col [V,T(x)]=0\;\forall x\in\Rl\}
	.
	\\
	&=
	\{\varphi(P)\,:\,\varphi\in\Inn_0(\Cl_+)\},
\end{align*}
where $\Inn_0(\Cl_+)$ denotes the semigroup of symmetric inner functions on $\Cl_+$, namely the set of all bounded holomorphic functions $\varphi:\Cl_+\to\Cl$  satisfying $|\varphi(p)|=1$ for almost all $p\in\Rl$ (in the sense of boundary values) and the symmetry condition $\varphi(-p)=\overline{\varphi(p)}$ for almost all $p\in\Rl$.

As a consequence of a theorem of Borchers \cite{Borchers:1992}, $J_HPJ_H=P$, and hence every element $V=\varphi(P)$ of $\CE(H,T)$ is symmetric in the sense $V^\#=J_HV^*J_H=V$ \eqref{eq:Uhash} introduced before. We are interested in the inclusions
\begin{align}
	\varphi H\subset H,\qquad \varphi\in\Inn_0(\Cl_+)
\end{align}
and their relative symplectic complements. Here and in the following, we write $\varphi H$ as a shorthand for $\varphi(P)H$.

To analyze these inclusions, it is helpful to have a concrete description of the semigroup of inner functions. This is essentially the classical Beurling factorization \cite{Rudin:1970} of inner functions $\varphi$ adapted to the symmetry requirement $\overline{\varphi(p)}=\varphi(-p)$, $p\in\Rl$, see \cite{LongoWitten:2010,Tanimoto:2015}.

\begin{lemma}
	A function $\varphi:\Cl_+\to\Cl$ lies in $\Inn_0(\Cl_+)$ if and only if it is of the form
	\begin{align}
		\varphi(p)
		=
		B_W(p)s_\nu(p),\qquad \Im p>0,
	\end{align}
	where i) $W\subset \Cl_+$ is a finite or infinite sequence $W=(p_k)_k$ such that $\sum_k\frac{\Im(p_k)}{1+|p_k|^2}<\infty$ and such that $w\in W\Leftrightarrow -\overline{w}\in W$ (taking multiplicity into account), and $B_W$ the Blaschke product with zero set $W$, namely\footnote{A general, not necessarily symmetric, Blaschke product for the upper half plane has the form $\prod_{w\in W}c_w\frac{p-w}{p-\overline{w}}$ with $c_w=\frac{|1+w^2|}{1+w^2}$ for $w\neq i$ and $c_i=1$. Because of the symmetry of $W$, these factors $c_w$ simplify in our case.}
	\begin{align}\label{eq:BlaschkeProduct}
		B_W(p)=\prod_{w\in W}\eps_w\frac{p-w}{p-\overline{w}},
		\qquad
		\eps_w:=
		\begin{cases}
			-1 & w\in i(1,\infty)\\
			+1  & \text{else}
		\end{cases}
		,
	\end{align}
	and ii) $\nu$ is a Borel measure on $\Rl\cup\{\infty\}$ singular w.r.t. Lebesgue measure such that $\int_0^\infty\frac{d\nu(\la)}{1+\la^2}<\infty$ and $\nu(-\la)=\nu(\la)$, and $s_\nu$ the corresponding singular inner function,
	\begin{align}\label{eq:singularfunction}
		s_\nu(p)=\exp\left(-i\int_\Rl\frac{1+\la p}{p-\la}\frac{d\nu(\la)}{1+\la^2}\right).
	\end{align}
	In case $\varphi$ is of this form, the zero set $W$ and the measure $\nu$ are uniquely determined by $\varphi$.
\end{lemma}

The inclusions $\varphi H\subset H$ are in most cases beyond the reach of the usual compactness or split criteria:

\begin{lemma}
	Let $\varphi\in\Inn_0(\Cl_+)$ with $\liminf\limits_{p\to0^+}|\varphi(ip)|>0$ or $\liminf\limits_{p\to\infty}|\varphi(ip)|>0$. Then $\varphi H\subset H$ does not satisfy modular compactness, and in particular is not split.
\end{lemma}
\begin{proof}
	Let $h\in H$ be a unit vector, and consider the sequence of unit vectors $k_n(p):=\varphi(p)(\Delta_H^{-in}h)(p)=\varphi(p)h(e^{2\pi n}p)$, which goes weakly to zero as $n\to\infty$. With $h'_n:=J_H k_n$, we have
	\begin{align*}
		\liminf\limits_{n\to\infty}|\langle h'_n,k_n\rangle|
		&=
		\liminf\limits_{n\to\infty}\|\Delta_H^{1/4}\varphi\Delta_H^{-in}h\|^2
		\\
		&=
		\liminf\limits_{n\to\infty}\|\Delta_H^{-in}\Delta_H^{1/4}\varphi(e^{-2\pi n}P) h\|^2
		\\
		&=
		\liminf\limits_{n\to\infty}\int_0^\infty \frac{dp}{p}\,|h(ip)|^2\,|\varphi(i\,e^{-2\pi n}p)|^2
		\\
		&>0,
	\end{align*}
	where we have used $\Delta_H^{in}P\Delta_H^{-in}=e^{-2\pi n}P$ and Fatou's Lemma together with the assumption $\liminf\limits_{p\to0}|\varphi(ip)|>0$. The claim now follows from Prop.~\ref{prop:notcompactEK}. In case $\liminf\limits_{p\to\infty}|\varphi(ip)|>0$, the proof is analogous (either replace $n$ by $-n$, or use Lemma~\ref{lemma:Y}).
\end{proof}

The assumption on the $\liminf$ is true for many inner functions, e.g. Blaschke products with zeros not accumulating at $0$ or $\infty$, or a singular inner function with measure $\nu$ having zero weight at $0$ or $\infty$. If $\varphi$ contains a factor $e^{iap-ib/p}$ with $a,b>0$, then modular compactness and even modular nuclearity holds -- this situation corresponds to viewing $H$ as a one-particle localization space for a Rindler wedge in a positive energy representation of $\Rl^2$ with mass gap $m:=\sqrt{ab}>0$ \cite{BuchholzLechner:2004}.

\medskip

Given the failure of the split property, we need other tools to understand the structure of the inclusions
\begin{align}
	\varphi H\subset H,\qquad \varphi\in\Inn_0(\Cl_+).
\end{align}
The subsequent theorem clarifies this question in most cases. The rest of the section will be devoted to proving it.

\begin{theorem}\label{thm:standardpairsrelativecomplements}
	Consider the irreducible non-degenerate standard pair in the upper half plane picture.
	\begin{enumerate}
		\item\label{item:translations} Let $\varphi(p)=e^{ipa}$ or $\varphi(p)=e^{-i\frac{a}{p}}$ with $a>0$. Then $\varphi H'\cap H$ is cyclic.

		\item\label{item:finiteblaschke} Let $\varphi=B_W$ be a finite symmetric Blaschke product \eqref{eq:BlaschkeProduct} (finite sequence $W$). Then
		\begin{align}
			\Cl[\varphi H'\cap H]
			&=
			\begin{cases}
				\{0\} & |W|=1,\\
				\Bigg\{\frac{p\cdot Q(p)}{\prod_{w\in W}(p-\overline{w})}
				\col
				Q\in\CP_{|W|-2}
				\Bigg\} & |W|>1
			\end{cases}
			.
		\end{align}
		where $\CP_{N-2}$ denotes the space of all polynomials of degree at most $N-2$. In particular, $\dim(B_WH'\cap H)=|W|-1$.

		\item\label{item:infiniteblaschke} Let $\varphi=B_W$ be an infinite symmetric Blaschke product. Then $\varphi H'\cap H$ is infinite-dimensional but not cyclic.

		\item\label{item:singular} Let $\varphi=s_\nu$ be a non-constant singular symmetric inner function. Then $\varphi H'\cap H$ is infinite-dimensional.
	\end{enumerate}
\end{theorem}

The different parts of this theorem require different techniques. We begin with part \ref{item:translations}, which is a known fact.

\medskip

\noindent
{\em Proof of Theorem~\ref{thm:standardpairsrelativecomplements}~\ref{item:translations}.} We consider $H$ in the form \eqref{eq:HbyFourierTransform} and realize that for $\varphi(p)=e^{ipa}$, the endomorphism $\varphi(P)=T(a)$ coincides with the translation unitary $T(a)$. Furthermore, the modular conjugation $J$ acts by $J\CF(\partial f)=\CF(\partial f_-)$, where $f_-(x)=f(-x)$, on test functions $f\in C_{c,\Rl}^\infty(\Rl)$. These remarks imply that the real linear space $$H_a:=\{\CF(\partial f)|_{[0,\infty)} \,:\, f\in C_{c,\Rl}^\infty([0,a])\}$$ is contained in $T(a)H'\cap H$. The space $H_a$ is cyclic for $a>0$, as can be seen by a Reeh-Schlieder argument \cite[Prop.~4.5]{LechnerLongo:2014}, so the claim follows for $\varphi(p)=e^{ipa}$. For $\varphi(p)=e^{-ia/p}$, we use the unitary $Y$ from Lemma~\ref{lemma:Y} to see that $\varphi H'\cap H=YT(a)YH'\cap H=YT(a)H'\cap YH=Y(T(a)H'\cap H)$ is cyclic as well. \hfill$\square$

\medskip

Also the proof of part \ref{item:finiteblaschke} is based on Fourier transforms.

\medskip

\noindent
{\em Proof of Theorem~\ref{thm:standardpairsrelativecomplements}~\ref{item:finiteblaschke}.} A finite symmetric Blaschke product $B_W$ \eqref{eq:BlaschkeProduct} is, up to a global prefactor $\pm1$, always of the form, $p\in\Rl$,
	\begin{align*}
		\varphi(p)
		=
		\frac{A^+(p)}{A^-(p)},\qquad
		A^+(p)=\prod_{w\in W}(-ip+iw),\quad
		A^-(p)=\prod_{w\in W}(-ip+i\overline{w}),
	\end{align*}
	where the product has $|W|$ factors and the parameters $w\in\Cl_+\setminus i\Rl$ occur in pairs $(w,-\overline{w})$, taking multiplicity into account. Due to this symmetry, $A^\pm_{\tt s}=A^\pm$.

	Let $h\in B_WH'\cap H$, i.e. there exists $h'\in H'$ such that
	\begin{align}
		A^+\cdot h'=A^-\cdot h.
	\end{align}
	Both sides of this equation exist as temperate distributions. As the action of $A^\pm$ corresponds to a finite order differential operator after Fourier transform, we have
	\begin{align*}
		\supp\CF^{-1}((A^+h')_{\tt s})
		&\subset
		\supp\CF^{-1}(h'_{\tt s})
		\subset
		(-\infty,0],\\
		\supp\CF^{-1}((A^-h)_{\tt s})
		&\subset
		\supp\CF^{-1}(h_{\tt s})
		\subset[0,\infty).
	\end{align*}
	This implies that there exists a polynomial $Q(p)=\sum_{m=0}^Mc_m\,p^m$ such that $Q(p)=(A^-h)(p)$ for all $p>0$, i.e.
	\begin{align}\label{eq:hQ}
		h(p)
		=
		\frac{Q(p)}{\prod_{w\in W}(-ip+i\overline{w})},\qquad p>0.
	\end{align}
	Analytically continuing to $\Rl_-$ we find the condition $\overline{Q(-p)}=Q(p)$, $p>0$, which is equivalent to $c_m\in\Rl$ for $m\in2\Nl_0$ and $c_m\in i\Rl$ for $m\in2\Nl_0+1$.

	The form of $Q$ is further constrained by the fact that $h$ has finite norm in $L^2(\Rl_+,\frac{dp}{p})$. The finiteness of $\int_0^\infty|h(p)|^2\frac{dp}{p}$ at $p\to\infty$ is equivalent to $\deg Q<|W|$, and the finiteness of this improper integral at $p\to0$ is equivalent to $0=\lim\limits_{p\to0}h(p)=\frac{c_0}{\prod_{w\in W}(i\overline{w})}$, i.e. $c_0=0$.

	Let us denote the set of all polynomials $Q$ satisfying these requirements $\CQ_\varphi$. For $|W|=1$, we arrive at $\CQ_\varphi=\{0\}$ and hence $B_WH'\cap H=\{0\}$, and for $|W|>1$ the space $\CQ_\varphi$ is a real linear space of dimension $|W|-1$.

	Given $Q\in\CQ_\varphi$, we see that $h$ (defined by \eqref{eq:hQ}) lies in $H$, and similarly $h'(p):=\frac{Q(p)}{\prod_{w\in W}(-ip+iw)}$ lies in $H'$. By construction, $B_Wh'=h$, so $h\in B_WH'$ and the proof is finished.\hfill$\square$

\medskip

\begin{remark}\label{remark:finiteblaschke}
	In part \ref{item:finiteblaschke} of this theorem, we have found an example of a singular inclusion coming from a unitary symmetric endomorphism, namely a Blaschke product $B(p)=\frac{p-ci}{p+ci}$, $c>0$, with a single zero. We also note that in the iterated inclusion $B^2H\subset BH\subset H$, both $B^2H'\cap BH$ and $BH'\cap H$ are $\{0\}$, but $B^2H'\cap H$ is one-dimensional (cf. the comment after Lemma~\ref{lemma:iteratedinclusions}~b)).
\end{remark}

It is interesting to compare our results for finite Blaschke products with those that can be obtained by the techniques of Section~\ref{section:extensions}, specifically the dimension bound $\dim(B_WH'\cap H)\geq n_U^+$ (Cor.~\ref{cor:DefIndBounds}), where $U=B_W$ and $n_U^+=\dim\ker(B_U^*-i)$ is a deficiency index of $B_U=U^*\Delta_H^{1/2}$. These deficiency indices were computed by Tanimoto in \cite{Tanimoto:2015} in his work on bound state operators. For us only part of this information is necessary, which we record below.

\begin{lemma}{\em\bf\cite[Thm.~4.1]{Tanimoto:2015}}\label{lemma:T}
	Let $\varphi=B_W$ be a symmetric Blaschke product and $n_W^\pm$ the deficiency indices of the bound state operator $B_W^*\Delta_H^{1/2}$. Then
	\begin{align}
			n_W^++n_W^-=|W|.
	\end{align}
\end{lemma}

In the case of a finite Blaschke product $B_W$ in which all zeros have even multiplicity (so in particular $|W|\geq2$), $U:=B_W(P)$ is the square of a symmetric endomorphism. Hence $n_W^+=n_W^-$ (Cor.~\ref{cor:DefIndBounds}), and the method of Cor.~\ref{cor:dimestimates} gives the estimate $\dim(B_WH'\cap H)\geq \frac{|W|}{2}$, to be compared with $\dim(B_WH'\cap H)=|W|-1$ by our explicit analysis Theorem~\ref{thm:standardpairsrelativecomplements}~\ref{item:finiteblaschke}. Hence the estimate based on deficiency indices is not optimal.

\bigskip

We now move on to part \ref{item:infiniteblaschke} of the theorem, concerning infinite symmetric Blaschke products $B_W$. In this case, we can consider a finite symmetric subset $W_0\subset W$ and write $B_W$ as
\begin{align*}
	B_W = B_{W_0}\cdot B_{W\setminus W_0}.
\end{align*}
Since $B_{W\setminus W_0}$ is also a symmetric inner function, we have
\begin{align*}
	B_{W}H'\cap H = B_{W_0}\cdot B_{W\setminus W_0}H'\cap H
	\supset
	B_{W_0}H'\cap H.
\end{align*}
As $W_0$ can be chosen as large was we want, and $\dim(B_{W_0}H'\cap H)=|W_0|-1$, we see that $B_{W}H'\cap H$ is infinite-dimensional, as claimed in part~\ref{item:infiniteblaschke}. It is however not so clear that it is not cyclic. To prove this, we will first apply Prop.~\ref{prop:limitendo}. The endomorphisms appearing there are here given by multiplying with symmetric inner functions:

\begin{lemma}\label{lemma:innerfunctionsconvergence}
	\leavevmode
	\begin{enumerate}
		\item\label{item:philimit} Let $(\varphi_n)_{n\in \Nl} \subset\Inn_0(\Cl_+)$ be a sequence of inner functions such that there exists $\varphi\in \Inn_0(\Cl_+)$ with $\lim_{n\to \infty}\varphi_n(p)=\varphi(p)$ for almost all $p\in \Rl$ and define $U_n\psi =\varphi_n \psi$ and $U\psi =\varphi\psi$. Then, $U_n\xrightarrow{SOT}U$.

		\item\label{item:blaschkelimit} Let $B_W$ be an infinite symmetric Blaschke product. Then there exist finite symmetric subsequences $W_N\subset W_{N+1}\subset W$ such that $W=\bigcup_N W_N$ and such that $\lim_{N\to\infty}B_{W_N}(p)= B_W(p)$ for almost all $p\in\Rl$.
	\end{enumerate}

\end{lemma}
\begin{proof}
	\ref{item:philimit} One easily sees that $U_n^\ast\psi =\overline{\varphi_n} \psi$ and $U^\ast\psi =\overline{\varphi} \psi$ are unitary endomorphisms because, for every $n\in\Nl$, $|\varphi_n(p)|=1=|\varphi(p)|$ for almost all $p\in \Rl$. Moreover, $\lim_{n\to \infty}\|U_n\psi-U\psi\|_2=\lim_{n\to\infty}\|(\varphi_n-\varphi)\psi\|_2=0,$ thanks to Lebesgue's Dominated Convergence Theorem. It means $U_n\to U$ strongly.

	\ref{item:blaschkelimit} Using the Cayley transform $p\xmapsto{\sigma} \frac{i-p}{i+p}$, we can unitarily map $L^2(\Rl,\tfrac{2dp}{|p+i|^2})$ onto $L^2(S^1)$. It is a straightforward calculation to show that a finite Blaschke product $B_{W_0}\in L^2(\Rl,\tfrac{2dp}{|p+i|^2})$ is mapped by the Cayley transform onto a finite Blaschke product $B^\Dl_{Z_0}:z\mapsto\prod_{\zeta\in Z_0}\eps^\Dl_z\frac{\zeta-z}{1-\bar\zeta z}
	\in L^2(S^1)$ for the disc, where $Z_0=\sigma(W_0)$ and $\eps_z^\Dl=-1$ for $-1<z<0$ and $\eps_z^\Dl=1$ otherwise. It is then also known that the partial products of Blaschke products on the disc, which we can be chosen symmetric, converge in $L^2(S^1)$  \cite[pag. 64-66]{Hoffman:2007} and \cite[Theorem 2.4.2]{Nikolskii:1986}. Hence, the sequence of symmetric partial products has a subsequence $B^\Dl_{Z_N}$ which converge almost everywhere on the disc to the Blaschke product $B^\Dl_Z$, where $Z=\sigma(W)$ is the set of zeros. Using the inverse Cayley transform, we conclude that the sequence of the boundary values of $B_{W_N}:=B^\Dl_{\sigma^{-1}(Z_N)}\circ\sigma^{-1}$ converges to $B_W$ almost everywhere.
\end{proof}

In view of this result, we may apply Proposition~\ref{prop:limitendo} to conclude
\begin{align}
	\Cl[B_W H' \cap H]
	=
	\Cl\left[\overline{\bigcup_{N\in \Nl} B_{W_N} H' \cap H}\right]
	\subset
	\overline{\bigcup_{N\in \Nl} \Cl[B_{W_N} H' \cap H]}.
\end{align}
To prove part~\ref{item:infiniteblaschke} of the theorem, it is therefore sufficient to show that the latter space is not dense. Note that we know from part~\ref{item:finiteblaschke} precisely what $\Cl[B_{W_N}H'\cap H]$ looks like, namely it consists of all functions of the form $p\mapsto pQ(p)\prod_{w\in W_N}(p-\overline{w})^{-1}$, where $W_N\subset W\subset\Cl_+$ is a finite symmetric sequence and $Q$ is an arbitrary polynomial of degree at most $|W_N|-2$.

The question of density of similar sets on $L^p$-spaces and continuous functions have been intensively studied in the past and we adapt here ideas from \cite[Addenda A]{Achieser:1956} to our case of interest. Since this is usually formulated on the unit disc $\Dl$ and not on a half plane, we will switch variables by a Möbius transform to $z=\frac{p+i}{p-i}\Leftrightarrow p=i\left(\frac{z+1}{z-1}\right)$, which maps the {\em lower} half plane onto $\Dl$. Furthermore, we will work with Hilbert spaces on $\Dl$ and write $|dz|$ for Lebesgue measure on $S^1=\partial\Dl$ and $dz$ for the line element of complex contour integrals.

We first notice that the real linear map $\psi\mapsto \psi_{\tt s}$ which associates to a function $\psi\in L^2(\Rl_+,\frac{dp}{p})$ its skew-symmetric extension defined in \eqref{eq:skew} is a surjective isometry from $L^2(\Rl_+,\frac{dp}{p})$ to the subspace of skew-symmetric functions of $L^2(\Rl,\frac{dp}{2|p|})$. Moreover, also the map  $\psi\mapsto \psi\circ \sigma$, where $z\xmapsto{\sigma}i\left(\frac{z+1}{z-1}\right):=p$, is a surjective isometry from the subspace of skew-symmetric functions of $L^2(\Rl,\frac{dp}{2|p|})$ to the subspace $L^2_{\tt s}(S^1,\frac{|dz|}{|z-1||z+1|})\subset L^2(S^1,\frac{|dz|}{|z-1||z+1|})$ consisting of functions satisfying the disc skew symmetry
\begin{align}\label{eq:skewdisc}
	\psi(z^{-1})=\overline{\psi(z)}.
\end{align}
Let us denote from now on by $U:L^2(\Rl_+,\frac{dp}{p})\to L^2_{\tt s}(S^1,\frac{|dz|}{|z-1||z+1|})$ the isometry given by $U(\psi)=\psi_{\tt s}\circ \sigma$. Then, a subspace $K\subset L^2(\Rl_+,\frac{dp}{p})$ being cyclic is equivalent to $U(K)\subset L^2_{\tt s}(S^1,\frac{|dz|}{|z-1||z+1|})$ being cyclic.

\begin{lemma}\label{lemma:unitarycircle}
	Let $W=(p_1,\ldots, p_{|W|})\subset \Cl\setminus\Rl$ be a finite sequence. The isometry $U$ maps $\Cl[B_{W}H'\cap H]$ into the subspace
	\begin{align}
		\B_Z
		:=
		\left\{
			z\mapsto \frac{(z+1)(z-1)Q(z)}{\prod_{\zeta\in Z}(z-\zeta)}
			\col
			Q \in \CP_{|Z|-2}
		\right\}
		\cap
		L^2_{\tt s}(S^1,\tfrac{|dz|}{|z-1||z+1|})
	\end{align}
	of $L^2_{\tt s}(S^1,\frac{|dz|}{|z-1||z+1|})$, where  $\zeta_n=\frac{\overline{p_n}+i}{\overline{p_n}-i}\in \Dl\setminus S^1$ and $Z=(\zeta_1,\ldots, \zeta_{|W|})$.
\end{lemma}
\begin{proof}
	The claim follows from a straightforward calculation using the Möbius transformation
	$z\xmapsto{\sigma}i\left(\frac{z+1}{z-1}\right):=p$. In fact, let $\psi\in L^2(\Rl_+,\frac{dp}{p})$ be of the form $\psi(p)=\frac{p Q(p)}{\prod_{w\in W}(p-\overline{w})}$, where $Q$ is a polynomial with $\deg(Q):=M\leq |W|-2$. We can factorize 
	$$Q(p)=C_0(p-i)^{M_1}\prod_{j=1}^{M_2} (p-q_j),$$
	where $M=M_1+M_2\leq |W|-2$, $C_0\in \Cl$, and $q_j\in \Cl\setminus\{i\}$, $1\leq j\leq M_2$. Then,
	\begin{align*}
		pQ(p)&=C_0p(p-i)^{M_1}\prod_{j=1}^{M_2} (p-q_j)\\
		&=\left(iC_0(2i)^{M_1}\prod_{j=1}^{M_2}(i-q_j)\right)(z-1)^{-1-M}(z+1)\prod_{j=1}^{M_2} (z-\zeta_j)\\
	\end{align*}
	where $\zeta_j=\frac{q_j+i}{q_j-i}\in \Cl$ for $1\leq n\leq M_2$. Notice that the coefficient between brackets never vanish and will be denoted by $\tilde{C}_0$.
	
	Doing the analogous transformation on the denominator, we can finally rewrite
	\begin{align*}
		U(\psi)(z):&=\frac{\tilde{C}_0(z+1)(z-1)(z-1)^{({|W|}-2-M)}\prod_{j=1}^{M_2} (z-\eta_j)}{\prod_{\zeta\in Z}(z-\zeta)}\\
		&=\frac{(z+1)(z-1)Q_1(z)}{\prod_{\zeta\in Z}(z-\zeta)},
	\end{align*}
	where $Q_1(z)=\tilde{C}_0(z-1)^{(|W|-2-M)}\prod_{j=1}^{M_2} (z-\eta_j)$ and $\zeta_n=\frac{\overline{p_n}+i}{\overline{p_n}-i}\in \Dl$ for $1\leq n\leq |W|$. We finish observing that $\deg(Q_1)=|W|-2-M+M_2=|W|-2-M_1\leq |W|-2$.
\end{proof}

\begin{proposition}\label{prop:distancBlaschke}
	Let $Z\subset \Dl\setminus S^1$ be a finite sequence. The distance of the set $\CB_Z\subset L^2(S^1,\frac{|dz|}{|(z-1)(z+1)|})$ to the function $f_0\in L_{\tt s}^2(S^1,\frac{|dz|}{|(z-1)(z+1)|})$ given by $f_0(z)=(z+1)(z-1)$ is larger than $2\pi$.
\end{proposition}
\begin{proof}
	It is elementary to check that the given function $f_0$ satisfies the symmetry condition \eqref{eq:skewdisc}.

	Let $\psi\in \CB_N$ be a function of the form $\psi(z)=\frac{(z+1)(z-1)Q(z)}{\prod_{\zeta\in Z}(z-\zeta)}$ and denote
	$I:=\|f_0-\psi\|_{L^2(S^1,\frac{|dz|}{|(z-1)(z+1)|})}^2$. We will estimate $I$ by techniques from approximation theory \cite{Achieser:1956}, including the elementary but very useful observations that for every $z\in S^1$ and $q\in \Cl$, we have $|z-q|=|1-\overline{q}z|$ and $|dz|=\frac{|dz|}{|z|}$, standard bounds on $L^2$-integrals by complex contour integrals, and the residue theorem. Denote by $q_n$, $1\leq n\leq|Z|$, the roots of the polynomial $$\prod_{\zeta\in Z}(z-\zeta)-Q(z)=\prod_{n=1}^{|Z|}(z-q_n)$$ ordered in such a way that $1\leq |q_n|$ for $1\leq n\leq M$.

	Combining these ideas,
	it follows that\begin{align*}
		I&=\int_{S^1}\left|1-\frac{Q(z)}{\prod_{\zeta\in Z}(z-\zeta)}\right|^2\left|(z-1)\left(z+1\right)\right|\, |dz|\\
		&=\int_{S^1}\left|\frac{\prod_{n=1}^{N}(z-q_n)}{\prod_{\zeta\in Z}(z-\zeta)}\right|^2\left|(z-1)(z+1)\right|\, |dz|\\
		&=\int_{S^1}\left|\frac{\prod_{n=1}^M(z-q_n)\prod_{n=M+1}^{N}(\overline{q_n}z-1)}{\prod_{\zeta\in Z}(1-\overline{\zeta}z)}\right|^2\left|\frac{(z-1)(z+1)}{z}\right|\, |dz|\\
		&\geq\left|\int_{S^1}\left(\frac{\prod_{n=1}^M(z-q_n)\prod_{n=M+1}^{N}(\overline{q_n}z-1)}{\prod_{\zeta\in Z}(1-\overline{\zeta}z)}\right)^2(z-1)(z+1)\, \frac{dz}{z}\right|\\
		&=\left|2\pi i\left(\prod_{n=1}^{M}q_n\right)^2 \right|\\
		&\geq  4\pi^2.
	\end{align*}
\end{proof}

\noindent
{\em Proof of Theorem~\ref{thm:standardpairsrelativecomplements}~\ref{item:infiniteblaschke}.} We have seen that $U$ maps $B_{W_N}H'\cap H$ into the space $\CB_{Z_N}$, which has distance larger $2\pi$ from the skew-symmetric function $f_0$ in the range of $U$, for all $N\in\Nl$, due to Lemma \ref{prop:distancBlaschke}. As the bound does not depend on $N$, it follows that $U^{-1}(f_0)$ does not lie in the closure of $\Cl[B_WH'\cap H]$.\hfill$\square$

\medskip

The proof of part~\ref{item:singular} relies on the results of Section~\ref{section:extensions}.

\noindent
{\em Proof of Theorem~\ref{thm:standardpairsrelativecomplements}~\ref{item:singular}.} Recall that any singular symmetric inner function $s_\nu$ is part of a one-parameter group of singular symmetric inner functions \cite{LongoWitten:2010}. Hence Corollary~\ref{cor:dimestimates} gives the claimed result. \hfill$\square$

\subsection*{Acknowledgements}
We are grateful to the Oberwolfach Research Institute for Mathematics for their hospitality during the mini-workshop ``Standard Subspaces in Quantum Field Theory and Representation Theory'' where part of this research was developed. G.L. also acknowledges financial support by the German Research Foundation DFG through the Heisenberg Grant ``Quantum Fields and Operator Algebras'' (LE 2222/3-1).

\AtNextBibliography{\small}
\newrefcontext[sorting=anyt]
\printbibliography
	
\end{document}